\documentclass[twoside]{article}

\usepackage{microtype}
\usepackage{graphicx}
\usepackage{subfigure}
\usepackage{booktabs} 
\usepackage{float}

\newcommand{\bound}{_\mathsf{\scriptscriptstyle b}}

\usepackage[accepted]{aistats2021}

\usepackage[round]{natbib}

\usepackage{hyperref}      

\hypersetup{ 
	pdfkeywords={},
	pdfborder=0 0 0,
	pdfpagemode=UseNone,
	colorlinks=true,
	linkcolor=blue, 
	citecolor=blue, 
	filecolor=blue, 
	urlcolor=blue, 
	pdfview=FitH,
	pdfauthor={Anonymous},
}

\usepackage{amsmath,amsfonts,amssymb,amsthm,array}
\usepackage{algorithm}
\usepackage{algorithmic}%
\usepackage{bm}
\usepackage{color}
\usepackage{graphicx}

\newcommand{\Exp}{\mathbb{E}}

\newcommand{\R}{\mathbb{R}}

\usepackage{multirow}

\newcommand{\bA}{\mathbf{A}}
\newcommand{\bB}{\mathbf{B}}

\newcommand{\bH}{\mathbf{H}}

\newcommand{\bS}{\mathbf{S}}

\newcommand{\bW}{\mathbf{W}}

\newcommand{\eqdef}{:=}

\newcommand{\cD}{{\cal D}}

\newcommand{\cL}{{\cal L}}

\newcommand{\cX}{{\cal X}}

\newcommand{\mA}{{\bf A}}

\newcommand{\mH}{{\bf H}}

\newcommand{\mS}{{\bf S}}

\usepackage{mdframed} 
\usepackage{thmtools}

\definecolor{shadecolor}{gray}{1.00}
\declaretheoremstyle[
headfont=\normalfont\bfseries,
notefont=\mdseries, notebraces={(}{)},
bodyfont=\normalfont,
postheadspace=0.5em,
spaceabove=1pt,
mdframed={
  skipabove=8pt,
  skipbelow=8pt,
  hidealllines=true,
  backgroundcolor={shadecolor},
  innerleftmargin=4pt,
  innerrightmargin=4pt}
]{shaded}

\declaretheorem[style=shaded,within=section]{definition}
\declaretheorem[style=shaded,sibling=definition]{theorem}

\declaretheorem[style=shaded,sibling=definition]{assumption}
\declaretheorem[style=shaded,sibling=definition]{corollary}

\declaretheorem[style=shaded,sibling=definition]{lemma}
\declaretheorem[style=shaded,sibling=definition]{remark}

\providecommand{\norm}[1]{\left\| #1\right\|}
\newcommand{\dotprod}[1]{\left< #1\right>}

\usepackage{hyperref}

\begin{document}

\runningtitle{Stochastic Polyak Step-size for SGD: An Adaptive Learning Rate for Fast Convergence}

\twocolumn[

\aistatstitle{Stochastic Polyak Step-size for SGD:  \\ An Adaptive Learning Rate for Fast Convergence}

\aistatsauthor{ Nicolas Loizou \And Sharan Vaswani$^\dagger$ \And  Issam Laradji \And Simon Lacoste-Julien}

\aistatsaddress{ Mila and DIRO \\ Universit\'{e} de Montr\'{e}al \And University of Alberta \And McGill, Element AI \And Mila and DIRO \\ Universit\'{e} de Montr\'{e}al \\ Canada CIFAR AI Chair }  ]

\begin{abstract}
We propose a stochastic variant of the classical Polyak step-size~\citep{polyak1987introduction} commonly used in the subgradient method. Although computing the Polyak step-size requires knowledge of the optimal function values, this information is readily available for typical modern machine learning applications. Consequently, the proposed stochastic Polyak step-size (SPS) is an attractive choice for setting the learning rate for stochastic gradient descent (SGD). We provide theoretical convergence guarantees for SGD equipped with SPS in different settings, including strongly convex, convex and non-convex functions. Furthermore, our analysis results in novel convergence guarantees for SGD with a constant step-size. We show that SPS is particularly effective when training over-parameterized models capable of interpolating the training data. In this setting, we prove that SPS enables SGD to converge to the true solution at a fast rate without requiring the knowledge of any problem-dependent constants or additional computational overhead. We experimentally validate our theoretical results via extensive experiments on synthetic and real datasets. We demonstrate the strong performance of SGD with SPS compared to state-of-the-art optimization methods when training over-parameterized models. 
\end{abstract}
\vspace{-4ex}
\section{Introduction}
\label{Intro}
We solve the finite-sum optimization problem:
\begin{equation}
\label{MainProb}
\min_{x\in\R^d} \left[ f(x) = \frac{1}{n} \sum_{i=1}^n f_i(x) \right].
\end{equation}
This problem is prevalent in machine learning tasks where~$x$ corresponds to the model parameters, $f_i(x)$ represents the loss on the training point $i$ and the aim is to minimize the average loss $f(x)$ across training points. We denote $\cX^* \subset \R^d$ to be the set of optimal points $x^*$ of~\eqref{MainProb} and assume that $\cX^*$ is not empty. We use $f^*$ to denote the minimum value of $f$, obtained at a point $x^* \in \cX^*$. For each $i \in \{1,\dots,n\}$, we denote the infimum of function $f_i$ by $f_i^* \eqdef \inf_{x} f_i(x)$. 
Depending on the model under study, the function $f$ can either be strongly-convex, convex, or non-convex.
\subsection{Background and Main Contributions}
Stochastic gradient descent (SGD) \citep{robbins1951stochastic, NemYudin1978, NemYudin1983book, Pegasos, Nemirovski-Juditsky-Lan-Shapiro-2009, HardtRechtSinger-stability_of_SGD}, is the workhorse for training supervised machine learning problems that have the generic form \eqref{MainProb}.  

\paragraph{Step-size selection for SGD.} The main parameter for guaranteeing the convergence of SGD is the \emph{step-size} or the \emph{learning rate}. In recent years, several ways of selecting the step-size have been proposed. \citet{moulines2011non,needell2014stochastic,batchSGDNW16, pmlr-v80-nguyen18c, gower2019sgd} propose a non-asymptotic analysis of SGD with \emph{constant step-size} for convex and strongly convex functions. For non-convex functions, such an analysis can be found in~\citet{ghadimi2013stochastic, bottou2018optimization}. Using a constant step-size for SGD guarantees convergence to a neighbourhoood of the solution. A common technique to guarantee convergence to the exact optimum is to use a \emph{decreasing step-size} \citep{robbins1951stochastic, ghadimi2013stochastic, gower2019sgd,Nemirovski-Juditsky-Lan-Shapiro-2009, karimi2016linear}. More recently, \emph{adaptive} methods \citep{duchi2011adaptive, liu2019variance, kingma2014adam, bengio2015rmsprop, vaswani2019painless, li2018convergence, ward2019adagrad} that adjust the step-size on the fly have become wide-spread and are particularly beneficial when training deep neural networks.

\emph{Contributions:} Inspired by the classical Polyak step-size \citep{polyak1987introduction} commonly used with the deterministic subgradient method \citep{hazan2019revisiting, boyd2003subgradient}, we propose a novel adaptive learning rate for SGD. The proposed step-size is a natural extension of the Polyak step-size to the stochastic setting. We name it \textbf{stochastic Polyak step-size (SPS)}. Although computing SPS requires knowledge of the $f_i^*$; we argue that this information is readily available for modern machine learning applications (for example, $f_i^*=0$ for most standard surrogate losses), making SPS an attractive choice for SGD.

In Section~\ref{SectionConvergenceAnalysis}, we provide theoretical guarantees for the convergence of SGD with SPS in different scenarios including strongly convex, convex and non-convex smooth functions. Although SPS is provably larger than the typically used constant step-size, we guarantee its convergence to a reasonable neighborhood around the optimum. We note that in the modern machine learning tasks that we consider, it is enough to converge to a small neighbourhood and not the exact minimizer to get good generalization performance. We also establish a connection between SPS and the optimal step-size used in sketch and project methods for solving linear systems. Furthermore, in Appendix~\ref{sec:additional-theory}, we provide convergence guarantees for convex non-smooth functions. We also show that by progressively increasing the batch-size for computing the stochastic gradients, SGD with SPS converges to the optimum. \\

\textbf{Technical assumptions and challenges for proving convergence.} Besides smoothness and convexity, several papers \citep{Shamir013, recht2011hogwild, hazan2014beyond, rakhlin2012making} assume that the variance of the stochastic gradient is bounded; that is there exists a $c$ such that $\Exp_{i}\|\nabla f_{i} (x)\|^2 \leq c$. However, in the unconstrained setting, this assumption contradicts the assumption of strong convexity \citep{pmlr-v80-nguyen18c, gower2019sgd}. In another line of work, growth conditions on the stochastic gradients have been used to guarantee convergence. In particular, the weak growth condition has been used in~\citet{Bertsekasneurodynamic,bottou2018optimization,pmlr-v80-nguyen18c}. It states that there exist constants $\rho, \delta$ such that $\Exp_{i}\|\nabla f_{i} (x)\|^2 \leq \rho \, \Exp\|\nabla f(x)\|^2 + \delta$. Its stronger variant (strong growth condition) when $\delta =0$ has been used in several recent papers \citep{schmidt2013fast, Volkan_Bang:2017, vaswani2018fast, vaswani2019painless}. These conditions can be relaxed to the expected smoothness assumption recently used in \citet{gower2019sgd}. 

\emph{Contributions:} Our analysis of SGD with SPS does not require any of these additional assumptions for guaranteeing convergence\footnote{Except for our analysis for non-convex smooth functions where the weak growth condition is used.}. We also note that our theoretical results do not require the finite-sum assumption and can be easily adapted to the streaming setting.

In addition, unlike standard analysis for constant step-size SGD, the use of SPS requires an adaptive step-size that uses the loss and stochastic gradient estimates at an iterate, resulting in correlations. One of the main technical challenges in the proofs is to carefully analyze the SGD iterates taking these correlations into account. Furthermore, since we need to be adaptive to the Lipschitz constant, we can not use the descent lemma (implied by smoothness and SGD
update). This makes the convex proof more challenging than the standard analysis.

\paragraph{Novel analysis for constant SGD.}
In the existing analyses of constant step-size SGD, the neighborhood of convergence depends on the variance of the gradients at the optimum, $z^2 \eqdef \Exp_i \norm{\nabla f_i(x^*)}^2$ which is assumed to be finite.

\emph{Contributions:} The proposed analysis of SGD with SPS gives a novel way to analyze constant step-size SGD. In particular, we prove convergence of \emph{constant step-size SGD} (without SPS), to a neighbourhood that depends on $\sigma^2 \eqdef f(x^*)-\Exp [ f_i^*] <\infty$ (finite optimal objective difference). 

\paragraph{Over-parametrized models and interpolation condition.}
Modern machine learning models such as non-parametric regression or over-parametrized deep neural networks are highly expressive and can fit or \emph{interpolate} the training dataset completely \citep{zhang2016understanding, ma2018power}. In this setting, SGD with constant step-size can been shown to converge to the exact optimum at the deterministic rate \citep{schmidt2013fast, ma2018power, vaswani2018fast, vaswani2019painless, gower2019sgd, berrada2019training}.

\emph{Contributions:} As a corollary of our  theoretical results, we show that SPS is particularly effective under this interpolation setting. Specifically, we prove that SPS enables SGD to converge to the true solution at a fast rate matching the deterministic case. Moreover, SPS does not require the knowledge of any problem-dependent constants or additional computational overhead. 

\paragraph{Experimental Evaluation.}
In Section~\ref{SectionExperiments}, we experimentally validate our theoretical results via experiments on synthetic datasets. We also evaluate the performance of SGD equipped with SPS relative to the state-of-the-art optimization methods when training over-parameterized models for deep matrix factorization, binary classification using kernels and multi-class classification using deep neural networks. For each of these tasks, we demonstrate the superior convergence of the proposed method. The code to reproduce our results can be found at \url{https://github.com/IssamLaradji/sps}.

\section{SGD and the Stochastic Polyak Step-size}
\label{SectionThePolyakStep}

The optimization problem \eqref{MainProb} can be solved using SGD: $$x^{k+1} =  x^k- \gamma_k \nabla f_{i}(x^k),$$
where example $i \in [n]$ is chosen uniformly at random and $\gamma_k>0$ is the step-size in iteration $k$.  
\subsection{The Polyak step-size}
Before explaining the proposed stochastic Polyak step-size, we first present the deterministic variant by Polyak  \citep{polyak1987introduction}. This variant is commonly used in the analysis of deterministic subgradient methods  \citep{boyd2003subgradient, hazan2019revisiting}. 

\paragraph{The deterministic Polyak step-size.} 
For convex functions, the deterministic Polyak step-size at iteration $k$ is the one that minimizes an upper-bound $Q(\gamma)$ on the distance of the iterate $x_{k+1}$ to the optimal solution: $ \|x^{k+1} - x^*\|_2^2 \leq Q(\gamma)$, where
$Q(\gamma) = \|x^k-x^*\|^2-2 \gamma \left[f(x^k)-f^*)\right] + \gamma^2 \| g^k\|^2.$
That is, $$\gamma_k=\text{argmin}_{\gamma} \left[Q(\gamma)\right] =\frac{f(x^k)-f^*}{\|g^k\|^2}.$$ Here $g^k$ denotes a subgradient of function $f$ at point $x^k$ and $f^*$ the optimum function value. For more details and a convergence analysis of the deterministic subgradient method, please check Appendix~\ref{AppendixDeterministicPolyak}.  
Note that the above step-size can be used only when the optimal value $f^*$ is known, however  \citet{boyd2003subgradient} demonstrate that $f^*=0$ for several applications (for example, finding a point in the intersection of convex sets, positive semidefinite matrix completion and solving convex inequalities).

\paragraph{Stochastic Polyak Step-size.} It is clear that using the deterministic Polyak step-size in the update rule of SGD is impractical. It requires the computation of the function value $f$ and its full gradient in each iteration. 

To avoid this, we propose the stochastic Polyak step-size (SPS) for SGD:
\begin{equation}
\label{SPLR}
\text{SPS:} \quad \boxed{\gamma_k =\frac{f_i(x^k)-f_i^*}{c \, \|\nabla f_i(x^k)\|^2}}
\end{equation}
Note that SPS requires the evaluation of only the stochastic gradient $\nabla f_i(x^k)$ and of the function $f_i(x^k)$ at the current iterate (quantities that can be computed in the update rule of SGD without further cost). However, it requires the knowledge of $f_i^*$. An important quantity in the step-size is the  parameter $0<c \in \R$ which can be set theoretically based on the properties of the function under study. For example, for strongly convex functions, one should select $c=1/2$ for optimal convergence.

In addition to SPS, in some of our convergence results we require its bounded variant:
\begin{equation}
\label{SPLRmax}
\text{SPS}_{\max}: \quad \boxed{\gamma_k = \min \left\{ \frac{f_i(x^k)-f_i^*}{c \|\nabla f_i(x^k)\|^2}, \gamma_{\bound} \right\}}
\end{equation}
Here $\gamma_{\bound}>0$ is a bound that restricts SPS from being very large and is essential to ensure convergence to a small neighborhood around the solution. If $\gamma_{\bound}=\infty$ then $\text{SPS}_{\max}$ is equivalent to SPS.

Though SPS and $\text{SPS}_{\max}$ require knowledge of $f_i^*$, this information is often readily available. For machine learning problems using standard \emph{unregularized} surrogate loss functions (e.g. squared loss for regression, logistic loss for classification), $f_i^* = 0$~\citep{bartlett2006convexity}. In the presence of an additional regularization term (e.g. $\ell_2$ regularization), $f_i^*$ can be obtained in closed form for these standard losses.  We emphasize that since $f_i^* = \inf_{x} f_i(x)$, the functions $f_i$ are not required to achieve the minimum. This is important when using loss functions such as the logistic loss for which the infimum is achieved at infinity~\citep{soudry2018implicit}. Furthermore, we note that the deterministic Polyak step-size requires knowledge of $f^*$ which is a much stronger assumption than the knowledge of $f_i^*$.

\paragraph{Closely related work.} We now briefly compare against the recently proposed stochastic variants of the Polyak step-size \citep{rolinek2018l4, oberman2019stochastic, berrada2019training}. In Section~\ref{SectionConvergenceAnalysis}, we present a detailed comparison of the theoretical convergence rates. 

In \citet{rolinek2018l4}, the \textit{L4 algorithm} has been proposed showing that a stochastic variant of the Polyak step for SGD achieves good empirical results for training neural networks. However it has \emph{no theoretical convergence guarantees.} The step-size is very similar to SPS \eqref{SPLR} but each update requires an online estimation of the $f_i^*$ which does not result in robust empirical performance and \emph{requires up to three hyper-parameters.}

\citet{oberman2019stochastic} use a different variant of the stochastic Polyak step-size:  $\gamma_k=\frac{2[f(x^k)-f^*]}{\Exp_i\|\nabla f_i(x^k)\|^2}$. This step-size requires  knowledge of the quantity $\Exp_i\|\nabla f_i(x^k)\|^2$ for all iterates $x^k$ and the evaluation of $f(x^k)$ in each step, making it impractical for finite-sum problems with large~$n$. Moreover, their theoretical results focus only on strongly convex smooth functions.

In the ALI-G algorithm proposed by~\citet{berrada2019training}, the step-size is set as: $\gamma_k=\min \left\{\frac{f_i(x^k)}{\|\nabla f_i(x^k)\|^2+\delta}, \eta \right\}$, where $\delta>0$ is a positive constant. Unlike our setting, their theoretical analysis relies on an $\epsilon$-interpolation condition. Moreover, the values of the parameter $\delta$ and $\eta$ that guarantee convergence heavily depend on the smoothness parameter of the objective $f$, limiting the method's practical applicability. In Section~\ref{SectionConvergenceAnalysis}, we show that as compared to \citet{berrada2019training}, the proposed method results in both better rates and a smaller neighborhood of convergence. For the case of over-parameterized models, our step-size selection guarantees convergence to the exact solution while the step proposed in \citet{berrada2019training} finds only an approximate solution that could be $\delta$ away from the optimum. In Section~\ref{SectionExperiments}, we also experimentally show that $\text{SPS}_{\max}$ results in better convergence than ALI-G.

\subsection{Optimal Objective Difference}
Unlike the typical analysis of SGD that assumes a finite gradient noise $z^2 \eqdef \Exp[\norm{\nabla f_i(x^*)}^2]$, in all our results, we assume a finite optimal objective difference. 
\begin{assumption}[Finite optimal objective difference] 
\begin{align}
\sigma^2 \eqdef \Exp_i[ f_i(x^*)-f_i^*] = f(x^*)-\Exp_i [ f_i^*] < \infty
\label{eq:sigma_def2}    
\end{align}
\end{assumption}
This is a very weak assumption. 
Moreover when \eqref{MainProb} is the training problem of an over-parametrized model such as a deep neural network or involves solving a consistent linear system or classification on linearly separable data, each individual loss function $f_i$ attains its minimum at $x^*$, and thus $f_i(x^*)-f_i^* =0.$ In this \emph{interpolation} setting, it follows that $\sigma=0$. 

\vspace{-2ex}
\section{Convergence Analysis}
\label{SectionConvergenceAnalysis}
\vspace{-2ex}
In this section, we present the main convergence results. For the formal definitions and properties of functions see Appendix~\ref{BasicDefinitions}. Proofs of all key results can be found in the Appendix~\ref{AppendixProofs}. 

\subsection{Upper and Lower Bounds of SPS}
If a function $g$ is $\mu$-strongly convex and $L$-smooth the following bounds hold:
$\frac{1}{2 L} \|\nabla g(x)\|^2 \leq g(x)-\inf_x g(x)\leq \frac{1}{2\mu} \|\nabla g(x)\|^2.$
Using these bounds and by assuming that the functions $f_i$ in problem \eqref{MainProb} are $\mu_i$-strongly convex and $L_i$-smooth, it is straight forward to see that SPS can be lower and upper bounded as follows:
\vspace{-1ex}
\begin{align}
\label{NewBounds}
\frac{1}{2 c L_{\max}} \leq \frac{1}{2 c L_{i}} \leq \gamma_{k}=\frac{f_{i}(x^k)-f_i^*}{c \|\nabla f_{i}(x^k)\|^2} \leq \frac{1}{2 c \mu_{i}},
\end{align}
\vspace{-1ex}
where $L_{\max}=\max \{L_i\}_{i=1}^n$.

\subsection{Sum of convex functions: strongly convex objective}
In this section, we assume that all components $f_i$ are convex functions and that the objective function $f$ is $\mu$-strongly convex.
\begin{theorem}
\label{TheoremStronglyConvex}
Let $f_i$ be $L_i$-smooth convex functions and assume that the objective function $f$ is $\mu$-strongly convex function. Then, SGD with $\text{SPS}_{\max}$ with $c\geq1/2$ converges as: 
\begin{align}
\label{StronglyConvexTheoremResult}
\Exp \|x^{k}-x^*\|^2 
\leq \left(1-\mu \alpha \right)^k  \|x^0-x^*\|^2 + \frac{2\gamma_{\bound} \sigma^2 }{\mu \alpha},\!\!
\end{align}
where $\alpha:=\min \{\frac{1}{2cL_{\max}},\gamma_{\bound}\}$ and $L_{\max}=\max \{L_i\}_{i=1}^n$ is the maximum smoothness constant. The best convergence rate and the tightest neighborhood are obtained for $c=1/2$.
\end{theorem}
Note that in Theorem~\ref{TheoremStronglyConvex}, we do not make any assumption on the value of the upper bound $\gamma_{\bound}$. However, it is clear that for convergence to a small neighborhood of the solution $x^*$ (unique solution for strongly convex functions) $\gamma_{\bound}$ should not be very large\footnote{Note that neighborhood $\frac{2\gamma_{\bound} \sigma^2} {\mu \alpha}$ has $\gamma_{\bound}$ in the numerator and for the case of large $\gamma_{\bound}$, $\alpha=\frac{1}{2cL_{\max}}$.}.

Another important aspect of Theorem~\ref{TheoremStronglyConvex} is that it provides convergence guarantees without requiring strong assumptions like bounded gradients or growth conditions. We do not use these conditions because SPS provides a natural bound on the norm of the gradients.
In the following corollaries we make additional assumptions to better understand the convergence of SGD with $\text{SPS}_{\max}$.

In our first corollary, we assume that our model is able to interpolate the data (each individual loss function $f_i$ attains its minimum at $x^*$). This condition is satisfied for unregularized least-squares regression on a realizable dataset, or when using the squared-hinge loss on a linearly-separable dataset. The interpolation assumption enables us to guarantee the convergence of SGD with SPS, without an upper-bound on the step-size ($\gamma_{\bound}=\infty$).
\begin{corollary}
\label{noaksoa}
Assume interpolation ($\sigma=0$) and let all assumptions of Theorem~\ref{TheoremStronglyConvex} be satisfied. SGD with SPS with $c=1/2$ converges as: $$\Exp \|x^{k}-x^*\|^2 \leq\left(1-\frac{\mu}{L_{\max}}\right)^k \|x^0-x^*\|^2.$$
\label{interpolationStrongConvex}
\end{corollary}

We compare the convergence rate in Corollary~\ref{noaksoa} to that of stochastic line search (SLS) proposed in~\citet{vaswani2019painless}. In similar setting, SLS achieves the slower linear rate $\max \left\{1-\frac{\bar{\mu} }{L_{\max}}, 1-\gamma_{\bound} \bar{\mu} \right\}$, where $\bar{\mu}=\sum_{i=1}^n\mu_i / n $ is the average strong-convexity of the finite sum.
In particular, according to Theorem 1 of~\citet{vaswani2019painless}, the convergence of SLS requires that at least one of the $f_i$'s is $\mu_i$-strongly convex implying that the objective function $f$ is strongly convex. This is a stronger assumption than the one we have in Theorem~\ref{TheoremStronglyConvex}. We also note that $\bar{\mu} \leq \mu$.

In~\citet{berrada2019training}, ALI-G is analyzed under the strong assumption that all functions $f_i$ are $\mu$-strongly convex and $L$-smooth. For detailed comparison of SPS with ALI-G, see Appendix~\ref{AliGcomparison}.

An interesting outcome of Theorem~\ref{TheoremStronglyConvex} is a novel analysis for SGD with a constant step-size. In particular, note that if the bound in $\text{SPS}_{\max}$ is selected to be $\gamma_{\bound} \leq \frac{1}{2 c L_{\max}}$, then using the lower bound of \eqref{NewBounds}, it can be easily shown that our method reduces to SGD with constant step-size $ \gamma_k=\gamma=\gamma_{\bound}\leq\frac{1}{2 c L_{\max}}$. In this case, we obtain the following convergence rate.
\begin{corollary}
Let all assumptions of Theorem~\ref{TheoremStronglyConvex} be satisfied. SGD with $\text{SPS}_{\max}$
with $c=1/2$ and $\gamma_{\bound}\leq\frac{1}{ L_{\max}}$ becomes SGD with constant step-size $\gamma\leq\frac{1}{L_{\max}}$ and converges as:
\begin{align*}
\Exp\|x^{k}-x^*\|^2
\leq \left(1- \mu \gamma \right)^k \|x^0-x^*\|^2 +  \frac{2\sigma^2}{\mu}.
\end{align*}
If we further assume interpolation ($\sigma=0$), the iterates of SGD with constant step-size $\gamma\leq\frac{1}{L_{\max}}$ satisfy: $$\Exp\|x^{k}-x^*\|^2 \leq \left(1- \mu \gamma \right)^k \|x^0-x^*\|^2.$$
\end{corollary}

To the best of our knowledge, this is the first result that shows convergence of constant step-size SGD to a neighborhood that depends on the optimal objective difference $\sigma^2$ \eqref{eq:sigma_def2} and not on the variance $z^2=\Exp[\norm{\nabla f_i(x^*)}^2$. Note that if we assume that all function $f_i$ are $\mu$-strongly convex and $L$-smooth functions then the two notions of variance satisfy the following connection: $\frac{1}{2L}z^2\leq \sigma^2 \leq\frac{1}{2\mu}z^2$. 

\subsection{Sum of convex functions}
Here, we derive the convergence rate when all component functions $f_i$ are convex without any strong convexity and obtain the following theorem.
\begin{theorem}
\label{TheoremConvex}
Assume that $f_i$ are convex, $L_i$-smooth functions. SGD with $\text{SPS}_{\max}$ with $c = 1$ converges as: 
\begin{align*}
\Exp \left[f(\bar{x}^K)-f(x^*)\right] 
\leq \frac{\|x^0-x^*\|^2}{\alpha \, K} + \frac{2\sigma^2 \gamma_{\bound}}{\alpha} .
\end{align*}
Here $\alpha=\min \left\{\frac{1}{2cL_{\max}},\gamma_{\bound}\right\}$ and $\bar{x}^K=\frac{1}{K}\sum_{k=0}^{K-1} x^k$. 
\end{theorem}
Analogous to the strongly-convex case, the size of the neighbourhood is proportional to $\gamma_{\bound}$. When interpolation is satisfied and $\sigma = 0$, we observe that the unbounded variant of SPS with $\gamma_{\bound} = \infty$ converges to the optimum at a $O(1/K)$ rate. This rate is faster than the rates in~\citet{vaswani2019painless, berrada2019training} and we refer the reader to the Appendix for a detailed comparison. As in the strongly-convex case, by setting $\gamma_{\bound} \leq \frac{1}{2 c L_{max}}$, we obtain the convergence rate obtained by constant step-size SGD. 

\subsection{Consistent Linear Systems}
\label{ConsistentSystems}
In~\citet{ASDA}, given the consistent linear system $\bA x=b$, the authors provide a stochastic optimization reformulation of the form \eqref{MainProb} which is equivalent to the linear system in the sense that their solution sets are identical. That is, the set of minimizers of the stochastic optimization problem $\cX^*$ is equal to the set of solutions of the stochastic linear system $\cL \eqdef \{x : \bA x=b \}$.  An interesting property of this stochastic optimization problem is that\footnote{For more details on the stochastic reformulation problem and its properties see Appendix~\ref{LinearSystemsDetails}.}:$ f_{i}(x)-f_i^* \overset{f_i^*=0}{=} f_{i}(x)= \frac{1}{2}\|\nabla f_{i}(x)\|^2 \quad \forall x\in \R^d.$
Using the special structure of the problem, SPS \eqref{SPLR} with $c=1/2$ takes the following form:
$\gamma_k \overset{\eqref{SPLR}}{=} \frac{2\left[f_i(x^k)-f_i^*\right]}{\|\nabla f_i(x^k)\|^2}=1,$
which is the theoretically optimal constant step-size for SGD in this setting \citep{ASDA}. This reduction implies that SPS results in an optimal convergence rate when solving consistent linear systems. We provide the convergence rate for SPS in this setting in Appendix~\ref{AppendixProofs}.

\subsection{Sum of non-convex functions: PL Objective}
We first focus on a special class of non-convex functions that satisfy the Polyak-Lojasiewicz (PL) condition~\citep{polyak1987introduction}. The PL inequality is a generalization of strong-convexity and is satisfied for matrix factorization~\citep{sun2016guaranteed} or when minimizing the logistic loss on a compact set~\citep{karimi2016linear}. In particular, we assume that function $f$ satisfies the PL condition but do not assume convexity of the component functions $f_i$. 

\begin{definition}[Polyak-Lojasiewicz (PL) condition]
We say that a function $f : \R^n \rightarrow \R$ satisfies the PL condition if there exists $\mu>0$ such that, $\forall x \in \R^n$ :
\begin{equation}
\label{PLPL}
\|\nabla f(x)\|^2 \geq 2 \mu (f(x)-f^*)
\end{equation}
\end{definition}

\begin{theorem}
\label{TheoremNonConvexPL}
Assume that function $f$ satisfies the PL condition~\eqref{PLPL}, and let $f$ and $f_i$ be smooth functions. SGD with $\text{SPS}_{\max}$ with $c > \frac{L_{\max}}{4\mu}$ and $\gamma_{\bound} \geq \frac{1}{2 c L_{\max}}$ converges as: 
\begin{align*}
\Exp[f(x^k)-f(x^{*})] \leq \nu^k \, [f(x^0)-f(x^{*})] + \frac{L\sigma^2\gamma_{\bound}}{2 (1-\nu) \, c}
\end{align*}
where $\nu= \gamma_{\bound} \left(\frac{1}{\alpha} - 2 \mu + \frac{L_{\max}}{2c}\right) \in (0,1]$ and $\alpha=\min \left\{\frac{1}{2cL_{\max}},\gamma_{\bound}\right\}$.  
\end{theorem}
Under the interpolation setting, $\sigma = 0$, and SPS$_{\max}$ converges to the optimal solution at a linear rate. If $\gamma_{\bound}\leq \min \left\{\frac{1}{2cL_{\max}},\frac{2c}{4 \mu c- L_{\max}}\right\}$ using the lower bound in \eqref{NewBounds}, the analyzed method is SGD with constant step-size and we obtain the following corollary. 
\begin{corollary}
\label{CorollaryConstantStepPL}
Assume that $f$ satisfies the PL condition~\eqref{PLPL}, and let $f$ and $f_i$ be smooth functions. SGD with constant step-size $\gamma_k=\gamma \leq\frac{\mu}{L_{\max}^2}$ converges as: 
$$
\Exp[f(x^k)-f(x^{*})] \leq \nu^k \, [f(x^0)-f(x^{*})] + \frac{L\sigma^2\gamma}{2 (1-\nu) \, c}.
$$
\end{corollary}

To the best of our knowledge this is the first result for the convergence of SGD for PL functions without assuming bounded gradient or bounded variance or interpolation (for more details see results in \citet{karimi2016linear} and discussion in \citet{gower2019sgd}). In the interpolation case, we obtain linear convergence to the optimum with a constant step-size equal to that used in~\citet{vaswani2018fast,  lei2019stochastic}.

\subsection{General Non-Convex Functions}
In this section, we assume a common condition used to prove convergence of SGD in the non-convex setting~\citep{bottou2018optimization}. 
\begin{equation}
\label{SGC}
\Exp[\|\nabla f_i(x)\|^2]\leq \rho \|\nabla f(x)\|^2 +\delta
\end{equation}
where $\rho, \delta>0$ constants.
\begin{theorem}
\label{TheoremNonConvex}
Let $f$ and $f_i$ be smooth functions and assume that there exist $\rho, \delta>0$ such that the condition \eqref{SGC} is satisfied. SGD with SPS$_{\max}$ with $c > \frac{\rho L}{4 L_{\max}}$ and $\gamma_{\bound} < \max \left\{\frac{2}{L\rho}, \bar{\gamma_{\bound}}\right\}$ converges as:
\begin{align*}
\min_{k \in [K]}\Exp \|\nabla f(x^k)\|^2 \leq  \frac{2}{\zeta K} \left(f(x^0)  -f(x^*)\right) \\ + \frac{\left(\gamma_{\bound}-\alpha+L \gamma_{\bound}^2\right)\delta}{\zeta},
\end{align*}
where $\alpha=\min\left\{\frac{1}{2cL_{\max}},\gamma_{\bound}\right\}, \quad \zeta=\left(\gamma_{\bound}+\alpha\right)-\rho\left(\gamma_{\bound}-\alpha+L \gamma_{\bound}^2\right)\quad$ and $$\bar{\gamma_{\bound}}\eqdef \frac{-(\rho-1)+\sqrt{(\rho-1)^2+\dfrac{4L\rho(\rho+1)}{2cL_{\max}}}}{2L\rho}.$$
\end{theorem}
From the above theorem, we observe that SGD with SPS results in $O(1/K)$ convergence to a neighborhoud governed by $\delta$. For the case that $\delta=0$, condition \eqref{SGC} reduces to the strong growth condition (SGC) used in several recent papers \citep{schmidt2013fast, vaswani2019painless,vaswani2018fast}. It can be easily shown that functions that satisfy the SGC condition necessarily satisfy the interpolation property~\citep{vaswani2018fast}. In the special case of interpolation, SGD with SPS is able to find a first-order stationary point as efficiently as deterministic gradient descent. Moreover, for $c \in \left(\frac{\rho L}{4 L_{\max}},\frac{\rho L}{2 L_{\max}}\right]$, the lower bound $\frac{1}{2cL_{\max}}$ of SPS lies in the range $\left[\frac{1}{\rho L}, \frac{2}{\rho L}\right)$ and thus the step-size is larger than $\frac{1}{\rho L}$, the best constant step-size analyzed in this setting~\citep{vaswani2018fast}. 

\subsection{Additional Convergence Results}
In Appendix~\ref{sec:additional-theory}, we prove a $O(1/\sqrt{K})$ convergence rate for non-smooth convex functions. Furthermore, similar to~\citet{schmidt201111}, we propose a method to increase the mini-batch size for evaluating the stochastic gradient and guarantee convergence to the optimal solution without interpolation. 

\section{Experimental Evaluation}
\label{SectionExperiments}
We validate our theoretical results using synthetic experiments in Section~\ref{sec:experiments-syn-ninter}. In Section~\ref{sec:exps-interpolation}, we evaluate the performance of SGD with SPS when training over-parametrized models. In particular, we compare against state-of-the-art optimization methods for deep matrix factorization, binary classification using kernel methods and multi-class classification using standard deep neural network models. 

\subsection{Synthetic experiments}
\label{sec:experiments-syn-ninter}
We use a synthetic dataset to validate our theoretical results. Following the procedure outlined in~\citet{nutini2017let}, we generate a sparse dataset for binary classification with the number of examples $n = 1$k and dimension $d = 100$. We use the logistic loss with and without $\ell_2$ regularization. The data is generated to ensure that the function $f$ is strongly convex in both cases. We evaluate the performance of SPS$_{\max}$ and set set $c = 1/2$ as suggested by Theorem~\ref{TheoremStronglyConvex}. We experiment with three values of $\gamma_{\bound} = \{1, 5, 100\}$. In the regularized case, $f^*_i$ can be pre-computed in closed form for each $i$ using the Lambert W function~\citep{corless1996lambertw} (see Appendix~\ref{app:compute:fi}); while $f_i^*$ is simply zero in the unregularized case. A similar observation has been used to construct a ``truncated'' model for improving the robustness of gradient descent in~\citet{asi2019importance}. In both cases, we benchmark the performance of SPS against constant step-size SGD with $\gamma = \{0.1, 0.01\}$. 
From Figure~\ref{fig:syn}, we observe that constant step-size SGD is not robust to the step-size; it has good convergence with step-size $0.1$, slow convergence when using a step-size of $0.01$ and we observe divergence for larger step-sizes. In contrast, all the variants of SPS converge to a neighbourhood of the optimum and the size of the neighbourhood increases as $\gamma_{\bound}$ increases as predicted by the theory. 

\begin{figure}[!h]
    \includegraphics[scale = 0.07]{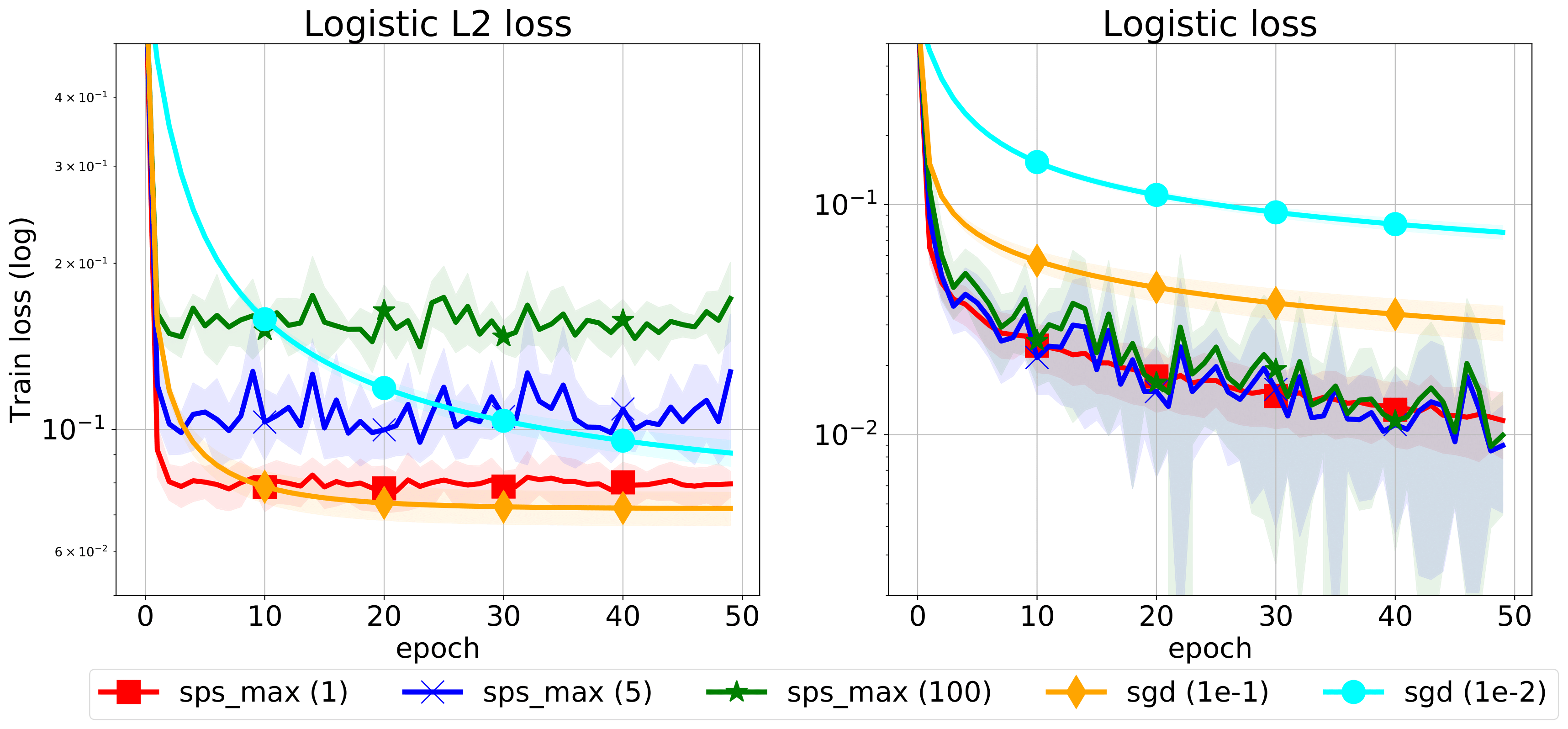}
    \caption{Synthetic experiment to benchmark SPS against constant step-size SGD for binary classification using the (left) regularized and (right) unregularized logistic loss.}
    \label{fig:syn}
\end{figure}

\subsection{Experiments for over-parametrized models}
\label{sec:exps-interpolation}

\begin{figure*}[!h]
    \centering
     \includegraphics[width = 1\textwidth]{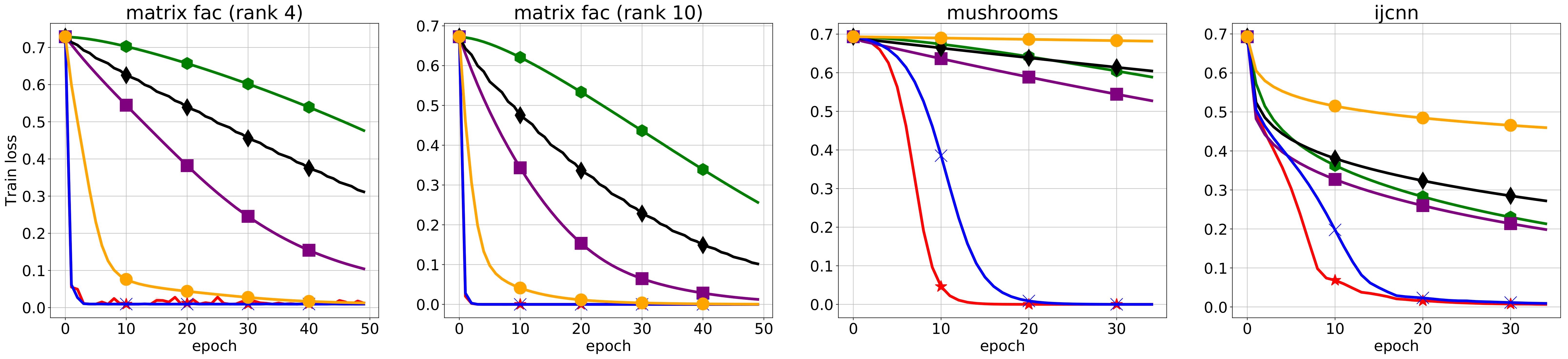}
     \includegraphics[width = 0.8\textwidth]{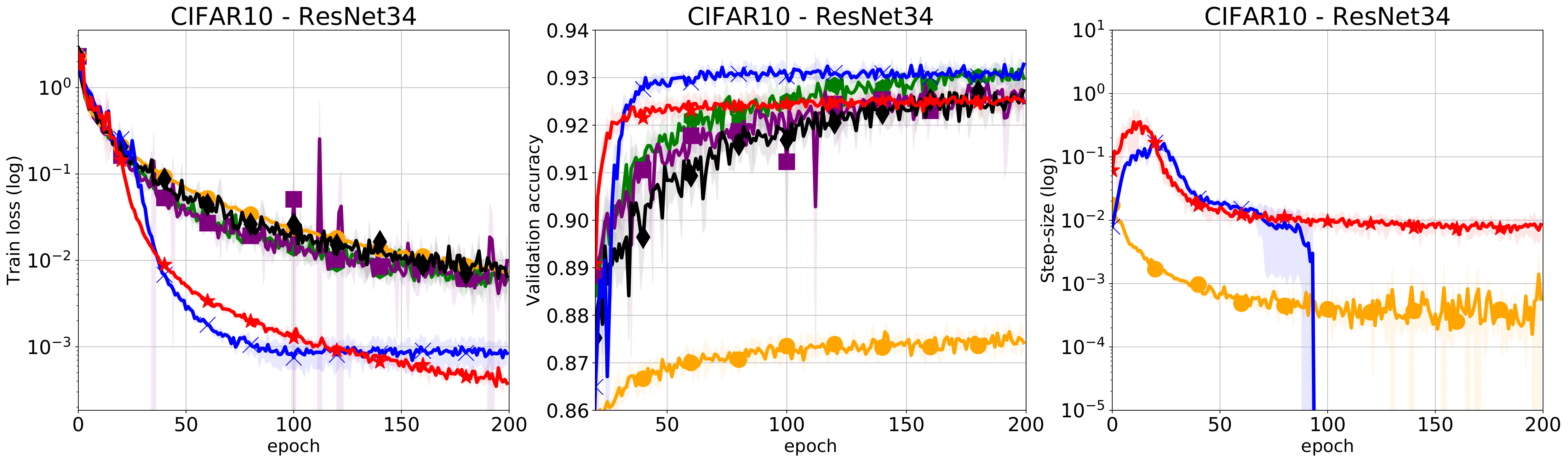}
    \includegraphics[width = 0.8\textwidth]{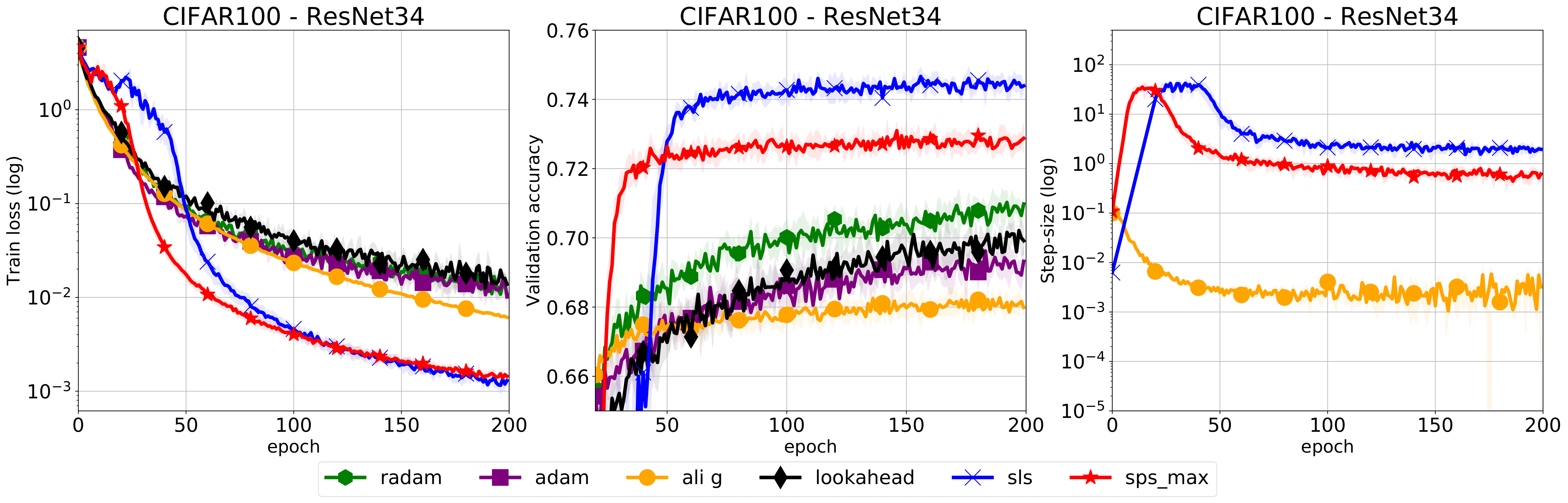}
    \caption{Comparing the performance of optimizers on deep matrix factorization (top left) and binary classification using kernels (top right) and multi-class classification on CIFAR-10 and CIFAR-100 with ResNet34.}
\label{fig:inter}    
\end{figure*}

In this section, we consider training over-parameterized models that (approximately) satisfy the interpolation condition. Following the logic of the previous section, we evaluate the performance of both the SPS and SPS$_{\max}$ variants with $f_i^* = 0$. Throughout our experiments, we found that SPS without an upper-bound on the step-size is not robust to the misspecification of interpolation and results in large fluctuations when interpolation is not exactly satisfied. For SPS$_{\max}$, the value of $\gamma_{\bound}$ that results in good convergence depends on the problem and requires careful parameter tuning. This is also evidenced by the highly variable performance of ALI-G~\citep{berrada2019training} that uses a constant upper-bound on the step-size. To alleviate this problem, we use a \emph{smoothing} procedure that prevents large fluctuations in the step-size across iterations. This can be viewed as using an adaptive iteration-dependent upper-bound $\gamma^{k}_{\bound}$ where $\gamma^{k}_{\bound} = \tau^{b/n} \,  \gamma^{k-1}$. Here, $\tau$ is a tunable hyper-parameter set to $2$ in all our experiments, $b$ is the batch-size and $n$ is the number of examples. We note that using an adaptive $\gamma_{\bound}$ can be easily handled by our theoretical results. A similar smoothing procedure has been used to control the magnitude of the step-sizes when using the Barzilai-Borwein step-size selection procedure for SGD~\citep{tan2016barzilai} and is related to the ``reset`` option for using larger step-sizes in~\citet{vaswani2019painless}. We set $c = 1/2$ for binary classification using kernels (convex case) and deep matrix factorization (non-convex PL case). For multi-class classification using deep networks, we empirically find that any value of $c \geq 0.2$ results in convergence. In this case, we observed that across models and datasets, the fastest convergence is obtained with $c = 0.2$ and use this value. 

We compare our methods against Adam~\citep{kingma2014adam}, which is the most common adaptive method, and other recent methods that report better performance than Adam: (i) stochastic line-search (SLS) \citep{vaswani2019painless} (ii) ALI-G~\citep{berrada2019training}\footnote{With ALI-G we refer to the method analyzed in~\citet{berrada2019training}. This is SGD with step-size the one described in Section~\ref{SectionThePolyakStep}. We highlight that the experiments in~\citet{berrada2019training} used momentum on top of the analyzed method but without any convergence quarantees. To ensure a fair comparison with SPS, we do not use such momentum.}  (iii) rectified Adam (RADAM)~\citep{liu2019variance} (iv) Look-ahead optimizer~\citep{zhang2019lookahead}. We use the default learning rates and momentum (non-zero) parameters and the publicly available code for the competing methods. All our results are averaged across $5$ independent runs. 

\paragraph{Deep matrix factorization.}
\label{sec:experiments-syn-inter}
In the first experiment, we use deep matrix factorization to examine the effect of over-parametrization for the different optimizers. In particular, we solve the non-convex regression problem: $\min_{W_1, W_2} \Exp_{x \sim N(0,I)} \norm{W_2 W_1 x - Ax}\kern-.1em{}^2$ and use the experimental setup in~\citet{rolinek2018l4, vaswani2019painless, rahimi2017reflections}. We choose $A \in \mathbb{R}^{10 \times 6}$ with condition number $\kappa(A) = 10^{10}$ and generate a fixed dataset of $1000$ samples. We control the degree of over-parametrization via the rank $k$ of the matrix factors $W_1 \in \mathbb{R}^{k \times 6}$ and $W_2 \in \mathbb{R}^{10 \times k}$. In Figure~\ref{fig:inter}, we show the training loss as we vary the rank $k \in \{4,10\}$ (additional experiments are in Appendix~\ref{app:additional}). For $k =  4$, the interpolation condition is \emph{not} satisfied, whereas it is exactly satisfied for $k = 10$. We observe that (i) SPS is robust to the degree of over-parametrization and (ii) has performance equal to that of SLS. However, note that SPS does not require the expensive back-tracking procedure of SLS and is arguably simpler to implement. 

\paragraph{Binary classification using kernels.}
\label{sec:experiments-kernels}
Next, we compare the optimizers' performance in the convex, interpolation regime. We consider binary classification using RBF kernels, using the logistic loss without regularization. The bandwidths for the RBF kernels are set according to the validation procedure described in~\citet{vaswani2019painless}. We experiment with four standard datasets: mushrooms, rcv1, ijcnn, and w8a from LIBSVM~\citep{libsvm}. Figure \ref{fig:inter} shows the training loss on the mushrooms and ijcnn for the different optimizers. Again, we observe the strong performance of SPS compared to the other optimizers. 

\paragraph{Multi-class classification using deep networks.}
We benchmark the convergence rate and generalization performance of SPS methods on standard deep learning experiments. We consider non-convex minimization for multi-class classification using deep network models on the CIFAR10 and CIFAR100 datasets. Our experimental choices follow the setup in~\citet{luo2019adaptive}. 
For CIFAR10 and CIFAR100, we experiment with the standard image-classification architectures: ResNet-34~\citep{he2016deep} and DenseNet-121~\citep{huang2017densely}. For space concerns, we report only the ResNet experiments in the main paper and relegate the DenseNet and MNIST experiments to Appendix~\ref{app:additional}. From Figure~\ref{fig:inter}, we observe that SPS results in the best training loss across models and datasets. For CIFAR-10, SPS results in competitive generalization performance compared to the other optimizers, whereas for CIFAR-100, its generalization performance is better than all optimizers except SLS. Note that ALI-G, the closest related optimizer results in worse generalization performance in all cases. We note that SPS is able to match the performance of SLS, but does not require an expensive back-tracking line-search or additional tricks. 

For this set of experiments, we also plot how the step-size varies across iterations for SLS, SPS and ALI-G. Interestingly, for both CIFAR-10 and CIFAR-100, we find that step-size for both SPS and SLS follows a cyclic behaviour - a warm-up period where the step-size first increases and then decreases to a constant value. Such a step-size schedule has been empirically found to result in good training and generalization performance~\citep{loshchilov2016sgdr} and our results show that SPS is able to simulate this behaviour. 

\vspace{-2ex}
\section{Conclusion}
\label{SectionConclusion}
\vspace{-2ex}
We proposed and theoretically analyzed a stochastic variant of the classical the Polyak step-size. We quantified the convergence rate of SPS in numerous settings and used our analysis techniques to prove new results for constant step-size SGD. Furthermore, via experiments on a variety of tasks we showed the strong performance of SGD with SPS as compared to state-of-the-art optimization methods. There are many possible interesting extensions of our work: using SPS with accelerated methods, studying the effect of mini-batching and non-uniform sampling techniques and extensions to the distributed and decentralized settings. 

\subsubsection*{Acknowledgements}
Nicolas Loizou and Sharan Vaswani acknowledge support by the IVADO Postdoctoral Funding Program. Issam Laradji is funded by the UBC Four-Year Doctoral Fellowships (4YF). This research was partially supported by the Canada CIFAR AI Chair Program and by a Google Focused Research award. Simon Lacoste-Julien is a CIFAR
Associate Fellow in the Learning in Machines \& Brains program.

The authors would like to thank Frederik Kunstner  for help with the convex proofs, and Aaron Defazio for fruitful discussions and feedback on the manuscript.

{
\bibliographystyle{apalike}
\bibliography{SPS}

\begin{thebibliography}{}

\bibitem[Asi and Duchi, 2019]{asi2019importance}
Asi, H. and Duchi, J.~C. (2019).
\newblock The importance of better models in stochastic optimization.
\newblock {\em Proceedings of the National Academy of Sciences},
  116(46):22924--22930.

\bibitem[Bartlett et~al., 2006]{bartlett2006convexity}
Bartlett, P.~L., Jordan, M.~I., and McAuliffe, J.~D. (2006).
\newblock Convexity, classification, and risk bounds.
\newblock {\em Journal of the American Statistical Association},
  101(473):138--156.

\bibitem[Bengio, 2015]{bengio2015rmsprop}
Bengio, Y. (2015).
\newblock Rmsprop and equilibrated adaptive learning rates for nonconvex
  optimization.
\newblock {\em corr abs/1502.04390}.

\bibitem[Berrada et~al., 2020]{berrada2019training}
Berrada, L., Zisserman, A., and Kumar, M.~P. (2020).
\newblock Training neural networks for and by interpolation.
\newblock In {\em ICML}.

\bibitem[Bertsekas and Tsitsiklis, 1996]{Bertsekasneurodynamic}
Bertsekas, D.~P. and Tsitsiklis, J.~N. (1996).
\newblock {\em Neuro-Dynamic Programming}.
\newblock Athena Scientific, 1st edition.

\bibitem[Bottou et~al., 2018]{bottou2018optimization}
Bottou, L., Curtis, F.~E., and Nocedal, J. (2018).
\newblock Optimization methods for large-scale machine learning.
\newblock {\em SIAM Review}, 60(2):223--311.

\bibitem[Boyd et~al., 2003]{boyd2003subgradient}
Boyd, S., Xiao, L., and Mutapcic, A. (2003).
\newblock Subgradient methods.
\newblock {\em lecture notes of EE392o, Stanford University, Autumn Quarter},
  2004:2004--2005.

\bibitem[Cevher and V{\~u}, 2019]{Volkan_Bang:2017}
Cevher, V. and V{\~u}, B.~C. (2019).
\newblock On the linear convergence of the stochastic gradient method with
  constant step-size.
\newblock {\em Optimization Letters}, 13(5):1177--1187.

\bibitem[Chang and Lin, 2011]{libsvm}
Chang, C.-C. and Lin, C.-J. (2011).
\newblock {LIBSVM}: A library for support vector machines.
\newblock {\em ACM Transactions on Intelligent Systems and Technology}.
\newblock Software available at \url{http://www.csie.ntu.edu.tw/~cjlin/libsvm}.

\bibitem[Corless et~al., 1996]{corless1996lambertw}
Corless, R.~M., Gonnet, G.~H., Hare, D.~E., Jeffrey, D.~J., and Knuth, D.~E.
  (1996).
\newblock On the {Lambert W} function.
\newblock {\em Advances in Computational mathematics}, 5(1):329--359.

\bibitem[Duchi et~al., 2011]{duchi2011adaptive}
Duchi, J., Hazan, E., and Singer, Y. (2011).
\newblock Adaptive subgradient methods for online learning and stochastic
  optimization.
\newblock {\em Journal of machine learning research}, 12(Jul):2121--2159.

\bibitem[Ghadimi and Lan, 2013]{ghadimi2013stochastic}
Ghadimi, S. and Lan, G. (2013).
\newblock Stochastic first-and zeroth-order methods for nonconvex stochastic
  programming.
\newblock {\em SIAM Journal on Optimization}, 23(4):2341--2368.

\bibitem[Gower and Richt{\'a}rik, 2015]{gower2015randomized}
Gower, R. and Richt{\'a}rik, P. (2015).
\newblock Randomized iterative methods for linear systems.
\newblock {\em SIAM Journal on Matrix Analysis and Applications},
  36(4):1660--1690.

\bibitem[Gower et~al., 2019]{gower2019sgd}
Gower, R.~M., Loizou, N., Qian, X., Sailanbayev, A., Shulgin, E., and
  Richt{\'a}rik, P. (2019).
\newblock {SGD}: General analysis and improved rates.
\newblock In {\em ICML}.

\bibitem[Hardt et~al., 2016]{HardtRechtSinger-stability_of_SGD}
Hardt, M., Recht, B., and Singer, Y. (2016).
\newblock Train faster, generalize better: stability of stochastic gradient
  descent.
\newblock In {\em ICML}.

\bibitem[Harikandeh et~al., 2015]{harikandeh2015stopwasting}
Harikandeh, R., Ahmed, M.~O., Virani, A., Schmidt, M., Kone{\v{c}}n{\`y}, J.,
  and Sallinen, S. (2015).
\newblock {Stop wasting my gradients: Practical SVRG}.
\newblock In {\em NeurIPS}.

\bibitem[Hazan and Kakade, 2019]{hazan2019revisiting}
Hazan, E. and Kakade, S. (2019).
\newblock Revisiting the polyak step size.
\newblock {\em arXiv preprint arXiv:1905.00313}.

\bibitem[Hazan and Kale, 2014]{hazan2014beyond}
Hazan, E. and Kale, S. (2014).
\newblock Beyond the regret minimization barrier: optimal algorithms for
  stochastic strongly-convex optimization.
\newblock {\em The Journal of Machine Learning Research}, 15(1):2489--2512.

\bibitem[He et~al., 2016]{he2016deep}
He, K., Zhang, X., Ren, S., and Sun, J. (2016).
\newblock Deep residual learning for image recognition.
\newblock In {\em CVPR}.

\bibitem[Huang et~al., 2017]{huang2017densely}
Huang, G., Liu, Z., Van Der~Maaten, L., and Weinberger, K.~Q. (2017).
\newblock Densely connected convolutional networks.
\newblock In {\em CVPR}.

\bibitem[Kaczmarz, 1937]{kaczmarz1937angenaherte}
Kaczmarz, S. (1937).
\newblock Angen{\"a}herte aufl{\"o}sung von systemen linearer gleichungen.
\newblock {\em Bulletin International de l’Academie Polonaise des Sciences et
  des Lettres}, 35:355--357.

\bibitem[Karimi et~al., 2016]{karimi2016linear}
Karimi, H., Nutini, J., and Schmidt, M. (2016).
\newblock Linear convergence of gradient and proximal-gradient methods under
  the {P}olyak-{{\L}}ojasiewicz condition.
\newblock In {\em ECML-PKDD}.

\bibitem[Kingma and Ba, 2015]{kingma2014adam}
Kingma, D. and Ba, J. (2015).
\newblock Adam: {A} method for stochastic optimization.
\newblock In {\em {ICLR}}.

\bibitem[Lei et~al., 2019]{lei2019stochastic}
Lei, Y., Hu, T., Li, G., and Tang, K. (2019).
\newblock Stochastic gradient descent for nonconvex learning without bounded
  gradient assumptions.
\newblock {\em IEEE transactions on neural networks and learning systems},
  31(10):4394--4400.

\bibitem[Li and Orabona, 2019]{li2018convergence}
Li, X. and Orabona, F. (2019).
\newblock On the convergence of stochastic gradient descent with adaptive
  stepsizes.
\newblock In {\em AISTATS}.

\bibitem[Liu et~al., 2019]{liu2019variance}
Liu, L., Jiang, H., He, P., Chen, W., Liu, X., Gao, J., and Han, J. (2019).
\newblock On the variance of the adaptive learning rate and beyond.
\newblock {\em arXiv preprint arXiv:1908.03265}.

\bibitem[Lohr, 2019]{lohr2019sampling}
Lohr, S.~L. (2019).
\newblock {\em {Sampling: Design and Analysis: Design and Analysis}}.
\newblock Chapman and Hall/CRC.

\bibitem[Loizou, 2019]{loizou2019randomized}
Loizou, N. (2019).
\newblock {\em Randomized iterative methods for linear systems: momentum,
  inexactness and gossip}.
\newblock PhD thesis, University of Edinburgh.

\bibitem[Loizou and Richt{\'a}rik, 2019]{loizou2019revisiting}
Loizou, N. and Richt{\'a}rik, P. (2019).
\newblock Revisiting randomized gossip algorithms: General framework,
  convergence rates and novel block and accelerated protocols.
\newblock {\em arXiv preprint arXiv:1905.08645}.

\bibitem[Loizou and Richt{\'a}rik, 2020a]{loizou2019convergence}
Loizou, N. and Richt{\'a}rik, P. (2020a).
\newblock Convergence analysis of inexact randomized iterative methods.
\newblock {\em SIAM Journal on Scientific Computing}, 42(6):A3979--A4016.

\bibitem[Loizou and Richt{\'a}rik, 2020b]{loizou2017momentum}
Loizou, N. and Richt{\'a}rik, P. (2020b).
\newblock Momentum and stochastic momentum for stochastic gradient, newton,
  proximal point and subspace descent methods.
\newblock {\em Computational Optimization and Applications}, 77(3):653--710.

\bibitem[Loshchilov and Hutter, 2017]{loshchilov2016sgdr}
Loshchilov, I. and Hutter, F. (2017).
\newblock {SGDR}: Stochastic gradient descent with warm restarts.
\newblock {\em ICLR}.

\bibitem[Luo et~al., 2019]{luo2019adaptive}
Luo, L., Xiong, Y., Liu, Y., and Sun, X. (2019).
\newblock Adaptive gradient methods with dynamic bound of learning rate.
\newblock In {\em ICLR}.

\bibitem[Ma et~al., 2018]{ma2018power}
Ma, S., Bassily, R., and Belkin, M. (2018).
\newblock The power of interpolation: Understanding the effectiveness of sgd in
  modern over-parametrized learning.
\newblock In {\em ICML}.

\bibitem[Mez{\H{o}} and Baricz, 2017]{mezHo2017generalization}
Mez{\H{o}}, I. and Baricz, {\'A}. (2017).
\newblock On the generalization of the {Lambert} {W} function.
\newblock {\em Transactions of the American Mathematical Society},
  369(11):7917--7934.

\bibitem[Moulines and Bach, 2011]{moulines2011non}
Moulines, E. and Bach, F.~R. (2011).
\newblock Non-asymptotic analysis of stochastic approximation algorithms for
  machine learning.
\newblock In {\em NeurIPS}.

\bibitem[Needell et~al., 2016]{needell2014stochastic}
Needell, D., Srebro, N., and Ward, R. (2016).
\newblock Stochastic gradient descent, weighted sampling, and the randomized
  kaczmarz algorithm.
\newblock {\em Mathematical Programming, Series A}, 155(1):549--573.

\bibitem[Needell and Ward, 2017]{batchSGDNW16}
Needell, D. and Ward, R. (2017).
\newblock Batched stochastic gradient descent with weighted sampling.
\newblock In {\em Approximation Theory XV, Springer}, volume 204 of {\em
  Springer Proceedings in Mathematics \& Statistics,}, pages 279 -- 306.

\bibitem[Nemirovski et~al., 2009]{Nemirovski-Juditsky-Lan-Shapiro-2009}
Nemirovski, A., Juditsky, A., Lan, G., and Shapiro, A. (2009).
\newblock Robust stochastic approximation approach to stochastic programming.
\newblock {\em SIAM Journal on Optimization}, 19(4):1574--1609.

\bibitem[Nemirovski and Yudin, 1978]{NemYudin1978}
Nemirovski, A. and Yudin, D.~B. (1978).
\newblock On {C}ezari's convergence of the steepest descent method for
  approximating saddle point of convex-concave functions.
\newblock {\em Soviet Mathetmatics Doklady}, 19.

\bibitem[Nemirovski and Yudin, 1983]{NemYudin1983book}
Nemirovski, A. and Yudin, D.~B. (1983).
\newblock {\em Problem complexity and method efficiency in optimization}.
\newblock Wiley Interscience.

\bibitem[Nguyen et~al., 2018]{pmlr-v80-nguyen18c}
Nguyen, L., Nguyen, P.~H., van Dijk, M., Richt\'{a}rik, P., Scheinberg, K., and
  Tak\'{a}\v{c}, M. (2018).
\newblock {SGD} and hogwild! {C}onvergence without the bounded gradients
  assumption.
\newblock In {\em ICML}.

\bibitem[Nutini et~al., 2017]{nutini2017let}
Nutini, J., Laradji, I., and Schmidt, M. (2017).
\newblock Let's make block coordinate descent go fast: Faster greedy rules,
  message-passing, active-set complexity, and superlinear convergence.
\newblock {\em arXiv preprint arXiv:1712.08859}.

\bibitem[Oberman and Prazeres, 2019]{oberman2019stochastic}
Oberman, A.~M. and Prazeres, M. (2019).
\newblock Stochastic gradient descent with polyak's learning rate.
\newblock {\em arXiv preprint arXiv:1903.08688}.

\bibitem[Polyak, 1987]{polyak1987introduction}
Polyak, B. (1987).
\newblock Introduction to optimization. translations series in mathematics and
  engineering.
\newblock {\em Optimization Software}.

\bibitem[Rahimi and Recht, 2017]{rahimi2017reflections}
Rahimi, A. and Recht, B. (2017).
\newblock Reflections on random kitchen sinks - arg min blog.

\bibitem[Rakhlin et~al., 2012]{rakhlin2012making}
Rakhlin, A., Shamir, O., and Sridharan, K. (2012).
\newblock Making gradient descent optimal for strongly convex stochastic
  optimization.
\newblock In {\em ICML}.

\bibitem[Recht et~al., 2011]{recht2011hogwild}
Recht, B., Re, C., Wright, S., and Niu, F. (2011).
\newblock Hogwild: A lock-free approach to parallelizing stochastic gradient
  descent.
\newblock In {\em NeurIPS}.

\bibitem[Richt{\'a}rik and Tak{\'a}c, 2020]{ASDA}
Richt{\'a}rik, P. and Tak{\'a}c, M. (2020).
\newblock Stochastic reformulations of linear systems: algorithms and
  convergence theory.
\newblock {\em SIAM Journal on Matrix Analysis and Applications},
  41(2):487--524.

\bibitem[Robbins and Monro, 1951]{robbins1951stochastic}
Robbins, H. and Monro, S. (1951).
\newblock A stochastic approximation method.
\newblock {\em The Annals of Mathematical Statistics}, pages 400--407.

\bibitem[Rolinek and Martius, 2018]{rolinek2018l4}
Rolinek, M. and Martius, G. (2018).
\newblock L4: Practical loss-based stepsize adaptation for deep learning.
\newblock In {\em NeurIPS}.

\bibitem[Schmidt et~al., 2011]{schmidt201111}
Schmidt, M., Kim, D., and Sra, S. (2011).
\newblock 11 projected {N}ewton-type methods in machine learning.
\newblock {\em Optimization for Machine Learning}, page 305.

\bibitem[Schmidt and Roux, 2013]{schmidt2013fast}
Schmidt, M. and Roux, N. (2013).
\newblock Fast convergence of stochastic gradient descent under a strong growth
  condition.
\newblock {\em arXiv preprint arXiv:1308.6370}.

\bibitem[Shalev-Shwartz et~al., 2007]{Pegasos}
Shalev-Shwartz, S., Singer, Y., and Srebro, N. (2007).
\newblock Pegasos: primal estimated subgradient solver for {SVM}.
\newblock In {\em ICML}.

\bibitem[Shamir and Zhang, 2013]{Shamir013}
Shamir, O. and Zhang, T. (2013).
\newblock Stochastic gradient descent for non-smooth optimization: Convergence
  results and optimal averaging schemes.
\newblock In {\em ICML}.

\bibitem[Soudry et~al., 2018]{soudry2018implicit}
Soudry, D., Hoffer, E., Nacson, M.~S., Gunasekar, S., and Srebro, N. (2018).
\newblock The implicit bias of gradient descent on separable data.
\newblock {\em The Journal of Machine Learning Research}, 19(1):2822--2878.

\bibitem[Strohmer and Vershynin, 2009]{RK}
Strohmer, T. and Vershynin, R. (2009).
\newblock A randomized {K}aczmarz algorithm with exponential convergence.
\newblock {\em J. Fourier Anal. Appl.}, 15(2):262--278.

\bibitem[Sun and Luo, 2016]{sun2016guaranteed}
Sun, R. and Luo, Z.-Q. (2016).
\newblock Guaranteed matrix completion via non-convex factorization.
\newblock {\em IEEE Transactions on Information Theory}, 62(11):6535--6579.

\bibitem[Tan et~al., 2016]{tan2016barzilai}
Tan, C., Ma, S., Dai, Y.-H., and Qian, Y. (2016).
\newblock Barzilai-borwein step size for stochastic gradient descent.
\newblock In {\em NeurIPS}.

\bibitem[Vaswani et~al., 2019a]{vaswani2018fast}
Vaswani, S., Bach, F., and Schmidt, M. (2019a).
\newblock Fast and faster convergence of sgd for over-parameterized models and
  an accelerated perceptron.
\newblock In {\em AISTATS}.

\bibitem[Vaswani et~al., 2019b]{vaswani2019painless}
Vaswani, S., Mishkin, A., Laradji, I., Schmidt, M., Gidel, G., and
  Lacoste-Julien, S. (2019b).
\newblock Painless stochastic gradient: Interpolation, line-search, and
  convergence rates.
\newblock In {\em NeurIPS}.

\bibitem[Ward et~al., 2019]{ward2019adagrad}
Ward, R., Wu, X., and Bottou, L. (2019).
\newblock Adagrad stepsizes: Sharp convergence over nonconvex landscapes.
\newblock In {\em ICML}.

\bibitem[Zhang et~al., 2017]{zhang2016understanding}
Zhang, C., Bengio, S., Hardt, M., Recht, B., and Vinyals, O. (2017).
\newblock Understanding deep learning requires rethinking generalization.
\newblock {\em ICLR}.

\bibitem[Zhang et~al., 2019]{zhang2019lookahead}
Zhang, M., Lucas, J., Ba, J., and Hinton, G.~E. (2019).
\newblock Lookahead optimizer: k steps forward, 1 step back.
\newblock In {\em NeurIPS}.

\end{thebibliography}
}

\appendix 

\onecolumn
\aistatstitle{Supplementary Material  \\ Stochastic Polyak Step-size for SGD: \\ An Adaptive Learning Rate for Fast Convergence}
The Supplementary Material is organized as follows: In Section~\ref{TechPrelAppenidx}, we provide the basic definitions mentioned in the main paper. We also present the convergence of deterministic subgradient method with the classical Polyak step-size. In Section~\ref{AppendixProofs} we present the proofs of the main theorems and in Section~\ref{sec:additional-theory} we provide additional convergence results. In Section~\ref{app:compute:fi}, we provide the closed form solutions for $f_i*$ for standard regularized binary surrogate losses. Finally, additional numerical experiments are presented in Section~\ref{app:additional}.

\section{Technical Preliminaries}
\label{TechPrelAppenidx}

\subsection{Basic Definitions}
\label{BasicDefinitions}
Let us present some basic definitions used throughout the paper.

\begin{definition}[Strong Convexity / Convexity]
The function $f : \R^n \rightarrow \R$, is $\mu$-strongly convex, if there exists a constant
$\mu > 0$ such that $\forall x, y \in \R^n$:
\begin{equation}
\label{stronglyconvex}
 f(x) \geq f(y)+ \dotprod{\nabla f(y) , x-y} + \frac{\mu}{2} \norm{x-y}^2
\end{equation}
for all $x \in \R^d$. If inequality \eqref{stronglyconvex} holds with $\mu=0$ the function $f$ is convex.
\end{definition}

\begin{definition}[Polyak-Lojasiewicz Condition]
The function $f : \R^n \rightarrow \R$, satisfies the Polyak-Lojasiewicz (PL) condition, if there exists a constant
$\mu > 0$ such that $\forall x \in \R^n$:
\begin{equation}
\label{PLcondition}
\|\nabla f(x)\|^2 \geq 2 \mu (f(x)-f^*) 
\end{equation}
\end{definition}

\begin{definition}[$L$-smooth]
The function $f : \R^n \rightarrow \R$, $L$-smooth, if there exists a constant
$L > 0$ such that $\forall x, y \in \R^n$:
\begin{equation}
\label{Smooth}
\|\nabla f(x)-\nabla f(y)\| \leq L \|x-y\|
\end{equation}
or equivalently:
\begin{equation}
\label{Smooth2}
 f(x) \leq f(y)+ \dotprod{\nabla f(y) , x-y} + \frac{L}{2} \norm{x-y}^2
\end{equation}
\end{definition}

\subsection{The Deterministic Polyak step-size}
\label{AppendixDeterministicPolyak}

In this section we describe the Polyak step-size for the subgradient method as presented in \citet{polyak1987introduction} for solving $\min_{x\in\R^d} f(x)$ where $f$ is convex, not necessarily smooth function.

Consider the subgradient method:
$$x^{k+1}=x^k-\gamma_k g^k,$$
where $\gamma_k$ is the step-size (learning rate) and $g^k$ is any subgradient of function $f$ at point $x^k$.

\begin{theorem}
\label{DeterministicPolyakTheorem}
Let $f$ be convex function.  Let $\gamma_k= \frac{f(x^k)-f(x^*)}{\|g^k\|^2}$ be the step-size in the update rule of subgradient method. Here $f(x^*)$ denotes the optimum value of function $f$.  Let $G>0$ such that $\| g^k\|^2< G^2$. Then,
$$ f^k_* -f(x^*) \leq \frac{G \|x^0-x^*\|}{\sqrt{k+1}}=O\left(\frac{1}{\sqrt{k}}\right),$$
where $f^k_*=\min \{f(x^i): i=0,1,\dots,k\}$.
\end{theorem}
\begin{proof}
\begin{eqnarray}
\label{nalksxalx}
\|x^{k+1}-x^*\|^2&=&\|x^k-\gamma_k g^k-x^*\|^2\notag\\
&=&\|x^k-x^*\|^2-2 \gamma_k \langle x^k-x^*, g^k \rangle + \gamma_k^2 \| g^k\|^2\notag\\
&\leq&\|x^k-x^*\|^2-2 \gamma_k \left[f(x^k)-f(x^*)\right] + \gamma_k^2 \| g^k\|^2
\end{eqnarray}
where the last line follows from the definition of subgradient:
$$f(x^*)\geq f(x^k)+\langle x^k-x^*, g^k \rangle$$
Polyak suggested to use the step-size:
\begin{equation}
\label{PolyakOriginalStep}
\boxed{\gamma_k= \frac{f(x^k)-f(x^*)}{\|g^k\|^2}}
\end{equation}
which is precisely the step-size that minimize the right hand side of \eqref{nalksxalx}. 
That is, $$\gamma_k= \frac{f(x^k)-f(x^*)}{\|g^k\|^2}=\text{argmin}_{\gamma_k} \left[\|x^k-x^*\|^2-2 \gamma_k \left[f(x^k)-f(x^*)\right] + \gamma_k^2 \| g^k\|^2 \right]$$.
By using this choice of step-size in \eqref{nalksxalx} we obtain:
\begin{eqnarray}
\|x^{k+1}-x^*\|^2
&\leq&\|x^k-x^*\|^2-2 \gamma_k \left[f(x^k)-f(x^*)\right] + \gamma_k^2 \| g^k\|^2\notag\\
&\overset{\eqref{PolyakOriginalStep}}{=}&\|x^k-x^*\|^2-\frac{\left[f(x^k)-f(x^*)\right]^2}{\| g^k\|^2} 
\end{eqnarray}
From the above note that $\|x^k-x^*\|^2$ is monotonic function.
Now using telescopic sum and by assuming $\| g^k\|^2< G^2$ we obtain:
\begin{eqnarray}
\|x^{k+1}-x^*\|^2
&\leq &\|x^0-x^*\|^2- \frac{1}{G^2} \sum_{i=0}^k \left[f(x^i)-f(x^*)\right]^2
\end{eqnarray}
Thus,
$$\frac{1}{G^2} \sum_{i=0}^k \left[f(x^i)-f(x^*)\right]^2\leq \|x^0-x^*\|^2-\|x^{k+1}-x^*\|^2\leq \|x^0-x^*\|^2$$
Let us define $f^k_*=\min \{f(x^i): i=0,1,\dots,k\}$ then:
$[f^k_* -f(x^*)]^2\leq \frac{G^2 \|x^0-x^*\|^2}{k+1}$
and $$ f^k_* -f(x^*) \leq \frac{G \|x^0-x^*\|}{\sqrt{k+1}}=O\left(\frac{1}{\sqrt{k}}\right)$$
\end{proof}
For more details and slightly different analysis check \citet{polyak1987introduction} and \citet{boyd2003subgradient}. In \citet{hazan2019revisiting} similar analysis to the above have been made for the deterministic gradient descent ($g^k=\nabla f(x^k)$) under several assumptions. (convex, strongly convex , smooth).

\section{Proofs of Main Results}
\label{AppendixProofs}
In this section we present the proofs of the main theoretical results presented in the main paper. That is, the convergence analysis of SGD with $\text{SPS}_{\max}$ and SPS under different combinations of assumptions on functions $f_i$ and $f$ of Problem~\eqref{MainProb}.

First note that the following inequality can be easily obtained by the definition of $\text{SPS}_{\max}$ \eqref{SPLRmax}:
\begin{equation}
\label{cakslaasasa}
\gamma_k^2 \|\nabla f_{i}(x^k)\|^2 \leq \frac{\gamma_k }{c}\left[f_i(x^k)-f_i^*\right]
\end{equation}
We use the above inequality in several parts of our proofs. It is the reason that we are able to obtain an upper bound of $\gamma_k^2 \|\nabla f_{i}(x^k)\|^2 $ without any further assumptions. For the case of SPS \eqref{SPLR},  inequality \eqref{cakslaasasa} becomes equality.

\subsection{Proof of Theorem~\ref{TheoremStronglyConvex}}

\begin{proof}
\begin{eqnarray*}
\|x^{k+1}-x^*\|^2 &=&\|x^k-\gamma_k \nabla f_i(x^k)-x^*\|^2\notag\\
&=&\|x^k-x^*\|^2-2 \gamma_k \langle x^k-x^*, \nabla f_i(x^k) \rangle + \gamma_k^2 \| \nabla f_i(x^k)\|^2\notag\\
&\overset{\eqref{cakslaasasa}}{\leq}&\|x^k-x^*\|^2-2 \gamma_k \langle x^k-x^*, \nabla f_i(x^k) \rangle +  \frac{\gamma_k }{c} \left[f_i(x^k)-f_i^*\right] \notag\\
&\overset{c\geq1/2}{\leq} & \|x^k-x^*\|^2-2 \gamma_k \langle x^k-x^*, \nabla f_i(x^k) \rangle +  2 \gamma_k \left[f_i(x^k)-f_i^*\right] \notag\\
&=& \|x^k-x^*\|^2-2 \gamma_k \langle x^k-x^*, \nabla f_i(x^k) \rangle +  2 \gamma_k \left[f_i(x^k)-f_i(x^*)+f_i(x^*)-f_i^*\right] \notag\\
&=& \|x^k-x^*\|^2+2 \gamma_k \left[ -\langle x^k-x^*, \nabla f_i(x^k) \rangle + f_i(x^k)-f_i(x^*) \right]+ 2 \gamma_k \left[f_i(x^*)-f_i^*\right]
\end{eqnarray*}
From convexity of functions $f_i$ it holds that $-\langle x^k-x^*, \nabla f_i(x^k) \rangle + f_i(x^k)-f_i(x^*)\leq0$, $\forall i \in [n]$. Thus, 
\begin{eqnarray*}
\|x^{k+1}-x^*\|^2 &=& \|x^k-x^*\|^2+2 \gamma_k \underbrace{\left[ -\langle x^k-x^*, \nabla f_i(x^k) \rangle + f_i(x^k)-f_i(x^*) \right]}_{\leq 0}+ 2 \gamma_k \underbrace{\left[f_i(x^*)-f_i^*\right]}_{\geq 0} \notag\\
&\overset{\eqref{NewBounds},\, \eqref{SPLRmax}}{ \leq} & \|x^k-x^*\|^2+2 \min\left\{\frac{1}{2cL_{\max}},\gamma_{\bound}\right\} \left[ -\langle x^k-x^*, \nabla f_i(x^k) \rangle + f_i(x^k)-f_i(x^*) \right]\notag\\ && + 2\gamma_{\bound} \left[f_i(x^*)-f_i^*\right] 
\end{eqnarray*}
By taking expectation condition on $x^k$
\begin{eqnarray*}
\Exp_i\|x^{k+1}-x^*\|^2 &\leq &  \|x^k-x^*\|^2+2 \min\left\{\frac{1}{2cL_{\max}},\gamma_{\bound}\right\} \left[ -\langle x^k-x^*, \nabla f(x^k) \rangle + f(x^k)-f(x^*) \right] \notag\\ && + 2\gamma_{\bound} \Exp_i\left[f_i(x^*)-f_i^*\right] \notag\\
&\overset{\eqref{eq:sigma_def2}}{=} &  \|x^k-x^*\|^2+2 \min\left\{\frac{1}{2cL_{\max}},\gamma_{\bound}\right\} \left[ -\langle x^k-x^*, \nabla f(x^k) \rangle + f(x^k)-f(x^*) \right] \notag\\ && + 2\gamma_{\bound} \sigma^2
\end{eqnarray*}
From strong convexity of the objective function $f$ we have that $f(x^k)-f(x^*)-\langle x^k-x^*, \nabla f(x^k) \rangle \leq -\frac{\mu}{2}\|x^{k}-x^*\|^2$. Thus, we obtain:
\begin{eqnarray*}
\label{eq:sc-inter1}
\Exp_i\|x^{k+1}-x^*\|^2 
&\leq & \left(1-\mu \min\left\{\frac{1}{2cL_{\max}},\gamma_{\bound}\right\} \right) \|x^k-x^*\|^2  + 2\gamma_{\bound} \sigma^2
\end{eqnarray*}
Taking expectations again and using the tower property:
\begin{eqnarray}
\label{eq:sc-inter}
\Exp\|x^{k+1}-x^*\|^2
&\leq &\left(1-\mu \min \left\{\frac{1}{2cL_{\max}},\gamma_{\bound}\right\} \right) \Exp\|x^k-x^*\|^2  + 2\gamma_{\bound} \sigma^2 
\end{eqnarray}

Recursively applying the above and summing up the resulting geometric series gives:
\begin{eqnarray*}
\Exp \|x^{k}-x^*\|^2 &\leq & \left(1-\mu \min \left\{\frac{1}{2cL_{\max}},\gamma_{\bound}\right\} \right)^k \|x^0-x^*\|^2 \notag\\ && + 2\gamma_{\bound} \sigma^2 \sum_{j=0}^{k-1}\left(1-\mu \min \left\{\frac{1}{2cL_{\max}},\gamma_{\bound}\right\} \right)^j
\notag\\
&\leq & \left(1-\mu \min \left\{\frac{1}{2cL_{\max}},\gamma_{\bound}\right\} \right)^k  \|x^0-x^*\|^2 + 2\gamma_{\bound} \sigma^2 \frac{1}{\mu \min \left\{\frac{1}{2cL_{\max}},\gamma_{\bound}\right\}}
\end{eqnarray*}
Let $\alpha=\min \left\{\frac{1}{2cL_{\max}},\gamma_{\bound}\right\}$ then,
\begin{eqnarray}
\Exp \|x^{k}-x^*\|^2 
&\leq & \left(1-\mu \alpha \right)^k  \|x^0-x^*\|^2 + \frac{2\gamma_{\bound} \sigma^2 }{\mu \alpha}
\end{eqnarray}
From definition of $\alpha$ is clear that having small parameter $c$ improves both the convergence rate $1-\mu \alpha$ and the neighborhood $\frac{2\gamma_{\bound} \sigma^2 }{\mu \alpha}$. Since we have the restriction $c\geq\frac{1}{2}$ the best selection would be $c=\frac{1}{2}$.
\end{proof}

\subsubsection{Comparison to the setting from~\citet{berrada2019training}}
\label{AliGcomparison}
In the next corollary, in order to compare against the results for ALI-G from~\citet{berrada2019training}, we make the strong assumption that all functions $f_i$ have the same properties. We note that such an assumption in the interpolation setting is quite strong and reduces the finite-sum optimization to minimization of a single function in the finite sum. 
\begin{corollary}
\label{nkalsa}
Let all the assumptions in Theorem~\ref{TheoremStronglyConvex} be satisfied and let all $f_i$ be $\mu$-strongly convex\footnote{This is a much stronger assumption than assuming that functions $f_i$ are convex and $f$ is strongly convex, which is the main assumption of Theorem~\ref{TheoremStronglyConvex}. Nevertheless, assuming that all $f_i$ are $\mu$-strongly convex functions implies that the objective function $f$ is $\mu$-strongly convex and that functions $f$ are convex. Thus, Theorem~\ref{TheoremStronglyConvex} still holds.} and $L$-smooth. SGD with $\text{SPS}_{\max}$ with $c=1/2$ converges as:
\begin{eqnarray*}
\Exp\|x^{k}-x^*\|^2
\leq \left(1-\frac{\mu}{L}\right)^k \|x^0-x^*\|^2 +  \frac{2\sigma^2 L}{\mu^2}.
\end{eqnarray*}
For the interpolated case ($\sigma=0$) we obtain the same convergence as Corollary~\ref{interpolationStrongConvex} with $L_{\max}=L$.
\end{corollary}

Note that, the result of Corollary~\ref{nkalsa} is obtained by substituting $\gamma_{\bound}\overset{\eqref{NewBounds}}{=}\frac{1}{2 c \mu}\overset{c=\frac{1}{2}}{=}{\frac{1}{\mu}}$ into \eqref{StronglyConvexTheoremResult}.

For the setting of Corollary~\ref{nkalsa}, ~\citet{berrada2019training} show the linear convergence to a much larger neighborhood than ours and with slower rate. In particular, their rate is $1-\frac{\mu}{8L}$ and the neighborhood is $\frac{8L}{\mu}(\frac{\epsilon}{L}+\frac{\delta}{4L^2}+\frac{\epsilon}{2\mu})$ where $\delta>2L\epsilon$ and $\epsilon$ is the $\epsilon$-interpolation parameter $\epsilon>\max_i[ f_i(x^*)-f_i^*]$ which by definition is bigger than $\sigma^2$. Under interpolation where $\sigma=0$, our method converges linearly to the $x^*$ while the algorithm proposed by~\citet{berrada2019training} still converges to a neighborhood that is proportional to the parameter $\delta$.

\subsection{Proof of Theorem~\ref{TheoremConvex}}
\begin{proof}
\begin{eqnarray}
\label{noiaks2}
\|x^{k+1}-x^*\|^2&=&\|x^k-\gamma_k \nabla f_i(x^k)-x^*\|^2\notag\\
&=&\|x^k-x^*\|^2-2 \gamma_k \langle x^k-x^*, \nabla f_i(x^k) \rangle + \gamma_k^2 \| \nabla f_i(x^k)\|^2\notag\\
&\overset{convexity}{\leq}&\|x^k-x^*\|^2-2 \gamma_k \left[f_i(x^k)-f_i(x^*)\right] + \gamma_k^2 \| \nabla f_i(x^k)\|^2\notag\\
&\overset{\eqref{cakslaasasa}}{\leq}&\|x^k-x^*\|^2-2 \gamma_k \left[f_i(x^k)-f_i(x^*)\right] + \frac{\gamma_k }{c}\left[f_i(x^k)-f_i^*\right]\notag\\
&=&\|x^k-x^*\|^2-2 \gamma_k \left[f_i(x^k)-f_i^*+f_i^*-f_i(x^*)\right] + \frac{\gamma_k }{c}\left[f_i(x^k)-f_i^*\right]\notag\\
&=&\|x^k-x^*\|^2- \gamma_k\left(2-\frac{1}{c}\right) \underbrace{\left[f_i(x^k)-f_i^*\right]}_{>0} + 2 \gamma_k \underbrace{\left[f_i(x^*)-f_i^*\right]}_{>0}
\end{eqnarray}

Let $\alpha=\min\left\{\frac{1}{2cL_{\max}},\gamma_{\bound}\right\}$ and recall that from the definition of 
$\text{SPS}_{\max}$ \eqref{SPLRmax} we obtain:
\begin{align}
\label{boundsaojsnda2}
\alpha \overset{\eqref{NewBounds},\, \eqref{SPLRmax}}{ \leq} \gamma_{k} \leq \gamma_{\bound}
\end{align}
From the above if $\alpha = \frac{1}{2cL_{\max}}$ then the step-size is in the regime of the stochastic Polyak step \eqref{NewBounds}. In the case that $\alpha = \gamma_{\bound}$ then the analyzed method becomes the constant step-size SGD with stepsize $\gamma_{k} =  \gamma_{\bound}$.

Since $c> \frac{1}{2}$ it holds that $\left(2-\frac{1}{c}\right) >0$. Using \eqref{boundsaojsnda2} into \eqref{noiaks2} we obtain:
\begin{eqnarray}
\|x^{k+1}-x^*\|^2&\leq&\|x^k-x^*\|^2- \gamma_k\left(2-\frac{1}{c}\right)\left[f_i(x^k)-f_i^*\right]+ 2 \gamma_k \left[f_i(x^*)-f_i^*\right] \notag\\
&\overset{\eqref{boundsaojsnda2}}{\leq}&\|x^k-x^*\|^2- \alpha \left(2-\frac{1}{c}\right) \left[f_i(x^k)-f_i^*\right] + 2 \gamma_{\bound} \left[f_i(x^*)-f_i^*\right]\notag\\
&=&\|x^k-x^*\|^2- \alpha \left(2-\frac{1}{c}\right) \left[f_i(x^k)-f_i(x^*)+f_i(x^*)-f_i^*\right] \notag\\ && + 2 \gamma_{\bound} \left[f_i(x^*)-f_i^*\right]\notag\\
&=&\|x^k-x^*\|^2- \alpha \left(2-\frac{1}{c}\right) \left[f_i(x^k)-f_i(x^*)\right]- \alpha \left(2-\frac{1}{c}\right) \left[f_i(x^*)-f_i^*\right] \notag\\ && + 2 \gamma_{\bound} \left[f_i(x^*)-f_i^*\right]\notag\\
&\leq&\|x^k-x^*\|^2- \alpha \left(2-\frac{1}{c}\right) \left[f_i(x^k)-f_i(x^*)\right]+ 2 \gamma_{\bound} \left[f_i(x^*)-f_i^*\right]
\end{eqnarray}
where in the last inequality we use that $\alpha \left(2-\frac{1}{c}\right) \left[f_i(x^*)-f_i^*\right]>0$.

By rearranging:
\begin{eqnarray}
\alpha \left(2-\frac{1}{c}\right) \left[f_i(x^k)-f_i(x^*)\right]
&\leq&\|x^k-x^*\|^2-\|x^{k+1}-x^*\|^2 + 2 \gamma_{\bound} \left[f_i(x^*)-f_i^*\right]
\end{eqnarray}

By taking expectation condition on $x^k$ and dividing by $\alpha \left(2-\frac{1}{c}\right) $:
\begin{eqnarray*}
\label{bcaiujskaoi7}
f(x^k)-f(x^*) &\leq& \frac{c}{\alpha (2c-1)} \left(\|x^k-x^*\|^2- \Exp_i\|x^{k+1}-x^*\|^2 \right)+ 2\gamma_{\bound}\frac{c}{\alpha (2c-1)}\sigma^2 
\end{eqnarray*}
Taking expectation again and using the tower property:
\begin{eqnarray*}
\label{bcaiujskaoi6}
\Exp[f(x^k)-f(x^*)]
&\leq& \frac{c}{\alpha (2c-1)} \left(\Exp\|x^k-x^*\|^2- \Exp\|x^{k+1}-x^*\|^2 \right)+ 2\gamma_{\bound}\frac{c}{\alpha (2c-1)}\sigma^2 
\end{eqnarray*}

Summing from $k=0$ to $K-1$ and dividing by $K$:
\begin{eqnarray}
 \frac{1}{K}\sum_{k=0}^{K-1} \Exp\left[f(x^k)-f(x^*)\right]
&=& \frac{c}{\alpha (2c-1)}\frac{1}{K}\sum_{k=0}^{K-1} \left(\Exp\|x^k-x^*\|^2 -\Exp\|x^{k+1}-x^*\|^2 \right) \notag\\ && + \frac{1}{K}\sum_{k=0}^{K-1} \frac{2c\gamma_{\bound}  \sigma^2}{\alpha(2c-1)} \notag\\
&=& \frac{2c}{\alpha (2c-1)}\frac{1}{K}\ \left[\|x^0-x^*\|^2-\Exp\|x^{K}-x^*\|^2\right]+\frac{2c\gamma_{\bound}  \sigma^2}{\alpha(2c-1)} \notag\\
&\leq& \frac{c}{\alpha (2c-1)}\frac{1}{K} \|x^0-x^*\|^2+\frac{2c\gamma_{\bound}  \sigma^2}{\alpha(2c-1)}
\end{eqnarray}

Let $\bar{x}^K=\frac{1}{K}\sum_{k=0}^{K-1} x^k$, then:
\begin{eqnarray*}
\label{nkoaclmadsmakl}
\Exp \left[f(\bar{x}^K)-f(x^*)\right] \overset{Jensen}{\leq} \frac{1}{K} \sum_{k=0}^{K-1} \Exp \left[f(x^k)-f(x^*)\right]
&\leq& \frac{c}{\alpha (2c-1)}\frac{1}{K} \|x^0-x^*\|^2+\frac{2c\gamma_{\bound}  \sigma^2}{\alpha(2c-1)}
\end{eqnarray*}
For $c=1$:
\begin{eqnarray}
\Exp \left[f(\bar{x}^K)-f(x^*)\right] 
&\leq& \frac{ \|x^0-x^*\|^2}{\alpha K}+\frac{2\gamma_{\bound}  \sigma^2}{\alpha} 
\end{eqnarray}
and this completes the proof. 
\end{proof}

At this point we highlight that $c=1$ is selected to simplify the expression of the upper bound in \eqref{nkoaclmadsmakl}. This is not the optimum choice (the one that makes the rate and the neighborhood of the upper bound smaller). In order to compute the optimum value of $c$ one needs to follow similar procedure to \citet{gower2019sgd} and \citet{needell2014stochastic}. In this case $c$ will depend on parameter $\sigma$ and the desired accuracy $\epsilon$ of convergence.

However as we show bellow having $c=1$ allows SGD with SPS to convergence faster than the ALI-G algorithm \citep{berrada2019training} and the SLS algorithm~\citep{vaswani2019painless} for the case of smooth convex functions.

\paragraph{Comparison with other methods} Similar to the strongly convex case let us compare the above convergence for smooth convex functions with the convergence rates proposed in~\citet{vaswani2019painless} and \citet{berrada2019training}.

For the smooth convex functions,~\citet{berrada2019training} show the linear convergence to a much larger neighborhood than ours and with slower rate. In particular, their rate is $\frac{1}{K}\left(\frac{2L}{1-\frac{2L\epsilon}{\delta}}\right)$ and the neighborhood is $\frac{\delta}{L(1-\frac{2L\epsilon}{\delta})}$ where $\delta>2L\epsilon$ and $\epsilon$ is the $\epsilon$-interpolation parameter $\epsilon>\max_i[ f_i(x^*)-f_i^*]$ which by definition is bigger than $\sigma^2$. Under interpolation where $\sigma=0$, our method converges with a $O(1/K)$ rate to the $x^*$ while the algorithm proposed by~\citet{berrada2019training} still converges to a neighborhood that is proportional to the parameter $\delta$.

In the interpolation setting our rate is similar to the one obtain for the stochastic line search (SLS) proposed in \citet{vaswani2019painless}. In particular in the interpolation setting, SLS achieves the following $O(1/K)$ rate $\Exp \left[f(\bar{x}^K)-f(x^*)\right] \leq \frac{ \max \{3 L_{\max}, 2/\gamma_{\bound}\}}{ K} \|x^0-x^*\|^2 $ which has slightly worse constants than SGD with SPS.
\subsection{SPS on Methods for Solving Consistent Linear Systems}
\label{LinearSystemsDetails}

Recently several new randomized iterative methods (sketch and project methods) for solving large-scale linear systems have been proposed \citep{ASDA,loizou2017momentum,loizou2019convergence,gower2015randomized}. The main algorithm in this literature is the celebrated randomized Kaczmarz (RK) method \citep{kaczmarz1937angenaherte, RK} which can be seen as special case of SGD for solving least square problems \citep{needell2014stochastic}.  In this area of research, it is well known that the theoretical best constant step-size for RK method is $\gamma=1$.

As we have already mentioned in Section~\ref{ConsistentSystems}, given the consistent linear system 
\begin{equation}
\label{THELinearSystem}
\bA x=b,
\end{equation}
 \citet{ASDA} provide a stochastic optimization reformulation of the form \eqref{MainProb} which is equivalent to the linear system in the sense that their solution sets are identical. That is, the set of minimizers of the stochastic optimization problem $\cX^*$ is equal to the set of solutions of the stochastic linear system $\cL \eqdef \{x : \bA x=b \}$.  

In particular, the stochastic convex quadratic optimization problem proposed in \citet{ASDA}, can be expressed as follows:
\begin{equation}
\label{StochReform} 
\min_{x\in \R^n} f(x) \eqdef \Exp_{\bS \sim \cD} {f_{\mS}(x)}.
\end{equation}
Here the expectation is over random matrices $\mS$ drawn from an arbitrary, user defined, distribution $\cD$ and $f_{\mS}$ is a stochastic convex quadratic function of a  least-squares type, defined as 
 \begin{equation}
 \label{f_s}
 f_{\mS}(x) \eqdef \frac{1}{2}\|\mA x - b\|_{\mH}^2 = \frac{1}{2}(\mA x - b)^\top \mH (\mA x - b).
 \end{equation} 
Function  $f_{\mS}$ depends on the matrix $\mA\in \R^{m\times n}$ and vector $b\in \R^m$ of the linear system \eqref{THELinearSystem} and on a random symmetric positive semidefinite matrix  
$\mH \eqdef  \mS (\mS^\top \mA \mA^\top \mS)^\dagger \mS^\top.$
By $\dagger$ we denote the Moore-Penrose pseudoinverse.  

For solving problem~\eqref{StochReform}, \citet{ASDA} analyze SGD with constant step-size:
\begin{equation}
\label{SGD_ILinearSystems}
x^{k+1} = x^k - \gamma\nabla f_{\mS_k}(x^k),
\end{equation}
where $\nabla f_{\mS_k}(x^k)$ denotes the gradient of function $f_{\mS_k}$. In each step the matrix $\mS_k$ is drawn from the given distribution $\cD$.

The above update of SGD is quite general and as explained by \citet{ASDA} the flexibility of selecting distribution $\cD$ allow us to  obtain different stochastic reformulations of the linear system \eqref{THELinearSystem} and different special cases of the SGD update. For example the celebrated randomized Kaczmarz (RK) method can be seen as special cases of the above update as follows:

\textit{Randomized Kaczmarz Method:} Let pick in each iteration the random matrix $\bS=e_i$ (random coordinate vector) with probability $p_i=\|\bA_{i:}\|^2 / \|\bA\|_F^2$. In this setup the update rule of SGD \eqref{SGD_ILinearSystems} simplifies to $$x^{k+1}=x^k -\omega \frac{\bA_{i:} x^k -b_i}{\|\bA_{i:}\|^2} \bA_{i:}^ \top$$
Many other methods like Gaussian Kacmarz, Randomized Coordinate Descent, Gaussian Decsent and their block variants can be cast as special cases of the above framework. 
For more details on the general framework and connections with other research areas we also suggest \citep{loizou2019revisiting,loizou2019randomized}.

\begin{lemma}[Properties of stochastic reformulation \citep{ASDA}]
For all $x \in \R^n$ and any $\bS \sim \cD$ it holds that:
\begin{equation}
\label{normbound}
f_{\bS}(x)-f_{\bS}(x^*)\overset{f_{\bS}^*=0}{=}f_{\bS}(x) =\frac{1}{2}\|\nabla f_{\bS}(x)\|^2_{\bB}=\frac{1}{2}\langle \nabla f_{\bS}(x),x-x^* \rangle_{\bB}.
\end{equation}
Let $x^*$ is the projection of vector $x$ onto the solution set $\cX^*$ of the optimization problem $\min_{x \in \R^n} f(x)$ (Recall that by the construction of the stochastic optimization problems we have that $\cX^*=\cL$). Then: 
\begin{equation}
\label{b3}
\frac{\lambda_{\min}^+(\bW)}{2} \|x-x^*\|^2_{\bB} \leq  f(x)  .
\end{equation}
 where $\lambda_{\min}^+$ denotes the minimum non-zero eigenvalue of matrix $\bW=\Exp[\bA^\top \bH \bA]$.
\end{lemma}

As we will see in the next Theorem, using the special structure of the stochastic reformulation \eqref{StochReform}, SPS \eqref{SPLR} with $c=1/2$ takes the following form:
$$\gamma_k \overset{\eqref{SPLR}}{=} \frac{2\left[f_{\bS}(x^k)-f_{\bS}^*\right]}{\|\nabla f_{\bS}(x^k)\|^2}\overset{\eqref{normbound}}{=}1,$$
which is the theoretically optimal constant step-size for SGD in this setting \citep{ASDA}. This reduction implies that SPS results in an optimal convergence rate when solving consistent linear systems. We provide the convergence rate for SPS in the next Theorem.

Though a straight forward verification of the optimality of SPS, we believe that this is the first time that SGD with adaptive step-size is reduced to constant step-size when is used for solving linear systems. SPS does that by obtaining the best convergence rate in this setting.

\begin{theorem}
\label{niaoklsaos}
Let $\bA x=b$ be a consistent linear system and let $x^*$ is the projection of vector $x$ onto the solution set $\cX^*=\cL$. Then the SGD with SPS \eqref{SPLR} with $c=1/2$ for solving the stochastic optimization reformulation \eqref{StochReform} satisfies:
 \begin{eqnarray}
 \label{convergenceSGD}
\Exp \|x^{k}-x^*\|^2
 \leq\left(1-\lambda_{\min}^+ (\bW) \right)^k \|x^0-x^*\|^2 
 \end{eqnarray}
 where $\lambda_{\min}^+$ denotes the minimum non-zero eigenvalue of matrix $\bW=\Exp[\bA^\top \bH \bA]$. 
 \end{theorem}

\begin{proof}
\begin{eqnarray}
\label{x_k_omega}
\|x^{k+1}-x^*\|^2 & \overset{\eqref{SGD_ILinearSystems}}{=}& \|x^k-\gamma_k \nabla f_{\bS_k}(x^k)-x^*\|^2\notag\\
&=& \|x^k-x^*\|^2-2 \gamma_k \langle x^k-x^*,\nabla f_{\bS_k}(x^k) \rangle +\gamma_k^2 \|\nabla f_{\bS_k}(x^k)\|^2
\end{eqnarray}
Let us select $\gamma_k$ such that the RHS of inequality~\eqref{x_k_omega} is minimized. That is, let us select:
$$\gamma_k= \frac{\langle x^k-x^*,\nabla f_{\bS_k}(x^k) \rangle}{\|\nabla f_{\bS_k}(x^k)\|^2}\overset{\eqref{normbound}}{=}\frac{2\left[f_{\bS}(x)-f_{\bS}(x^*)\right]}{\|\nabla f_{\bS_k}(x^k)\|^2}$$

Substitute this step-size to \eqref{x_k_omega} we obtain:
\begin{eqnarray}
\|x^{k+1}-x^*\|^2
&\leq& \|x^k-x^*\|^2-2 \frac{\langle x^k-x^*,\nabla f_{\bS_k}(x^k) \rangle}{\|\nabla f_{\bS_k}(x^k)\|^2} \langle x^k-x^*,\nabla f_{\bS_k}(x^k) \rangle \notag\\ && +\left[\frac{\langle x^k-x^*,\nabla f_{\bS_k}(x^k) \rangle}{\|\nabla f_{\bS_k}(x^k)\|^2}\right]^2 \|\nabla f_{\bS_k}(x^k)\|^2\notag\\
&=& \|x^k-x^*\|^2-2 \frac{ [\langle x^k-x^*,\nabla f_{\bS_k}(x^k) \rangle]^2}{\|\nabla f_{\bS_k}(x^k)\|^2}  +\frac{\left[\langle x^k-x^*,\nabla f_{\bS_k}(x^k) \rangle\right]^2}{\|\nabla f_{\bS_k}(x^k)\|^2}\notag\\
&=& \|x^k-x^*\|^2- \frac{ [\langle x^k-x^*,\nabla f_{\bS_k}(x^k) \rangle]^2}{\|\nabla f_{\bS_k}(x^k)\|^2} \notag\\
&\overset{\eqref{normbound}}{=}& \|x^k-x^*\|^2- 2f_{\bS}(x^k)
\end{eqnarray}

By taking expectation with respect to $\mS_k$ and using quadratic growth inequality \eqref{b3}:
\begin{eqnarray}
\label{exp_x_k_omega}
\Exp_{\mS_k}[\|x^{k+1}-x^*\|^2] &=& \|x^k-x^*\|^2-2f(x^k)\notag\\
&\overset{\eqref{b3}}{\leq}& \|x^k-x^*\|^2- \lambda_{\min}^+(\bW)  \|x^k-x^*\|^2\notag\\
&=&\left[1-\lambda_{\min}^+(\bW) \right] \|x^k-x^*\|^2.
\end{eqnarray}
Taking expectation again and by unrolling the recurrence we obtain \eqref{convergenceSGD}.
\end{proof}
We highlight that the above proof provides a different viewpoint on the analysis of the optimal constant step-size for the sketch and project methods for solving consistent liner systems. The expression of Theorem~\ref{niaoklsaos} is the same with the one proposed in~\citet{ASDA}.

\subsection{Proof of Theorem~\ref{TheoremNonConvexPL}}
\begin{proof}
By the smoothness of function $f$ we have that 
$$f(x^{k+1}) \leq f(x^k) + \langle\nabla f(x^k) , x^{k+1}-x^k \rangle + \frac{L}{2}\|x^{k+1}-x^k\|^2.$$
Combining this with the update rule of SGD we obtain:
\begin{eqnarray}
f(x^{k+1})& \leq & f(x^k) + \left\langle\nabla f(x^k) , x^{k+1}-x^k \right\rangle + \frac{L}{2}\|x^{k+1}-x^k\|^2\notag\\
& = & f(x^k) - \gamma_k \left\langle\nabla f(x^k) , \nabla f_i(x^k) \right\rangle + \frac{L \gamma_k^2}{2}\|\nabla f_i(x^k)\|^2 
\end{eqnarray}

By rearranging:
\begin{eqnarray*}
\frac{f(x^{k+1})-f(x^k) }{\gamma_k}& \leq & - \left\langle\nabla f(x^k) , \nabla f_i(x^k) \right\rangle + \frac{L \gamma_k}{2}\|\nabla f_i(x^k)\|^2 \notag\\
& \overset{\eqref{cakslaasasa}}{\leq} &- \left\langle\nabla f(x^k) , \nabla f_i(x^k) \right\rangle + \frac{L}{2c}\left[ f_i(x^k)-f_i^*\right]\notag\\
&= &- \left\langle\nabla f(x^k) , \nabla f_i(x^k) \right\rangle + \frac{L}{2c}\left[ f_i(x^k)- f_i(x^*)\right]+\frac{L}{2c}\left[ f_i(x^*)-f_i^*\right]
\end{eqnarray*}
and by taking expectation condition on $x^k$:
\begin{eqnarray*}
\Exp_i\left[\frac{f(x^{k+1})-f(x^k) }{\gamma_k}\right]& \leq & - \left\langle\nabla f(x^k) , \nabla f(x^k) \right\rangle + \frac{L}{2c}\left[ f(x^k)- f(x^*)\right]+\frac{L}{2c}\Exp_i\left[ f_i(x^*)-f_i^*\right] \notag\\
&\overset{\eqref{eq:sigma_def2}}{\leq}&- \|\nabla f(x^k)\|^2 + \frac{L}{2c}\left[ f(x^k)- f(x^*)\right]+\frac{L}{2c}\sigma^2
\notag\\
&\overset{\eqref{PLcondition}}{\leq}&- 2\mu \left[ f(x^k)- f(x^*)\right] + \frac{L}{2c}\left[ f(x^k)- f(x^*)\right]+\frac{L}{2c}\sigma^2
\end{eqnarray*}
Let $\alpha= \min\left\{\frac{1}{2cL_{\max}},\gamma_{\bound}\right\}$. Then,
\begin{eqnarray}
\Exp_i\left[\frac{f(x^{k+1})-f(x^*) }{\gamma_k}\right]& \leq &\Exp_i\left[\frac{f(x^{k})-f(x^*) }{\gamma_k}\right]- 2\mu \left[ f(x^k)- f(x^*)\right] \notag\\ && + \frac{L}{2c}\left[ f(x^k)- f(x^*)\right]+\frac{L}{2c}\sigma^2\notag\\
&\overset{\eqref{NewBounds},\eqref{SPLRmax}}{\leq}&\frac{1}{\alpha}\left[f(x^{k})-f(x^*) \right]- 2\mu \left[ f(x^k)- f(x^*)\right] \notag\\ && + \frac{L}{2c}\left[ f(x^k)- f(x^*)\right]+\frac{L}{2c}\sigma^2\notag\\
&=&\left(\frac{1}{\alpha}- 2\mu  + \frac{L}{2c}\right)\left[ f(x^k)- f(x^*)\right]+\frac{L}{2c}\sigma^2\notag\\
&\overset{L\leq L_{\max}}{=}&\left(\frac{1}{\alpha}- 2\mu  + \frac{L_{\max}}{2c}\right)\left[ f(x^k)- f(x^*)\right]+\frac{L}{2c}\sigma^2
\end{eqnarray}
Using $\gamma_k\leq\gamma_{\bound}$ and by taking expectations again:
\begin{eqnarray}
\label{oakslxa}
\Exp\left[f(x^{k+1})-f(x^*)\right]& \leq & \underbrace{\gamma_{\bound} \left(\frac{1}{\alpha}- 2\mu  + \frac{L_{\max}}{2c}\right)}_{\nu}\Exp \left[ f(x^k)- f(x^*)\right]+\frac{L\sigma^2\gamma_{\bound}}{2c}
\end{eqnarray}

By having  $\nu \in (0, 1]$ and by recursively applying the above and summing the resulting geometric series we obtain:
\begin{eqnarray}
\Exp\left[f(x^k)-f(x^*)\right]& \leq & \nu^k  \left[ f(x^0)- f(x^*)\right]+\frac{L\sigma^2\gamma_{\bound}}{2c} \sum_{j=0}^{k-1}\nu^j\notag\\
& \leq &\nu^k  \left[ f(x^0)- f(x^*)\right]+\frac{L\sigma^2\gamma_{\bound}}{2(1-\nu)c}
\end{eqnarray}

In the above result we require that $0 < \nu=\gamma_{\bound}\left(\frac{1}{\alpha}- 2\mu  + \frac{L_{\max}}{2c}\right) \leq 1$. In order for this to hold we need to make extra assumptions on the values of $\gamma_{\bound}$ and parameter $c$. This is what we do next.

Let us divide the analysis into two cases based on the value of parameter $\alpha$. That is:
\begin{itemize}
\item(i) If $\frac{1}{2cL_{\max}}\leq \gamma_{\bound}$ then, 
$$\alpha= \min\left\{\frac{1}{2cL_{\max}},\gamma_{\bound}\right\}=\frac{1}{2cL_{\max}} \quad \text{and} \quad \nu=\gamma_{\bound}\left((2c+\frac{1}{2c})L_{\max}- 2\mu\right).$$
By preliminary computations, it can be easily shown that $\nu>0$ for every $c\geq0$. However for $\nu\leq1$ we need to require that  $\gamma_{\bound} \leq \frac{1}{\left(\frac{1}{\alpha}- 2\mu  + \frac{L_{\max}}{2c}\right)}$ and since we are already assume that $\frac{1}{2cL_{\max}}\leq \gamma_{\bound}$ we need to force 
$$\frac{1}{2cL_{\max}}\leq\frac{1}{\left(\frac{1}{\alpha}- 2\mu  + \frac{L_{\max}}{2c}\right)}$$
to avoid contradiction. This is true only if $c>\frac{L_{\max}}{4\mu}$ which is the assumption of Theorem~\ref{TheoremNonConvexPL}.
\item (ii) If $ \gamma_{\bound}\leq\frac{1}{2cL_{\max}}$ then, 
$$\alpha= \min\left\{\frac{1}{2cL_{\max}},\gamma_{\bound}\right\}=\gamma_{\bound} \quad \text{and} \quad \nu=\gamma_{\bound}\left(\frac{1}{\gamma_{\bound}}- 2\mu  + \frac{L_{\max}}{2c}\right)=1- 2\mu\gamma_{\bound}  + \frac{L_{\max}}{2c}\gamma_{\bound}.$$
Note that if we have $c>\frac{L_{\max}}{4\mu}$ (an assumption of Theorem~\ref{TheoremNonConvexPL}) it holds that $\nu\leq1$. In addition, by preliminary computations, it can be shown that $\nu>0$ if  $\gamma_{\bound} < \frac{2c}{4\mu c -L_{\max}}$.  Finally, for $c>\frac{L_{\max}}{4\mu}$ it holds that $\frac{1}{2cL_{\max}}\leq \frac{2c}{4\mu c -L_{\max}}$, and as a result $\nu>0$ for all $ \gamma_{\bound}\leq\frac{1}{2cL_{\max}}$.
\end{itemize}
By presenting the above cases on bound of $\nu$ we complete the proof.
\end{proof}

\begin{remark}
The expression of Corollary~\ref{CorollaryConstantStepPL} is obtained by simply use $c=\frac{L_{\max}}{2\mu}$ in the case (ii) of the above proof. In this case we have $\gamma \leq\frac{\mu}{L_{\max}^2}$ and $\nu=1-\mu \gamma$.
\end{remark}

\subsection{Proof of Theorem~\ref{TheoremNonConvex}}
\begin{proof}
First note that:
\begin{eqnarray}
\label{iaokalkdao1}
- \gamma_k \left\langle\nabla f(x^k) , \nabla f_i(x^k) \right\rangle&=&\frac{\gamma_k}{2} \|\nabla f_i(x^k)-\nabla f(x^k)\|^2 -\frac{\gamma_k}{2}  \|\nabla f_i(x^k)\|^2 -\frac{\gamma_k}{2}  \|\nabla f(x^k)\|^2 \notag\\
& \overset{\eqref{boundsaojsnda2}}{\leq} & \frac{\gamma_{\bound}}{2} \|\nabla f_i(x^k)-\nabla f(x^k)\|^2 -\frac{\alpha}{2}  \|\nabla f_i(x^k)\|^2 -\frac{\alpha}{2}  \|\nabla f(x^k)\|^2 \notag\\
&=& \frac{\gamma_{\bound}}{2} \|\nabla f_i(x^k)\|^2+\frac{\gamma_{\bound}}{2} \|\nabla f(x^k)\|^2 - \gamma_{\bound}\left\langle\nabla f(x^k) , \nabla f_i(x^k) \right\rangle \notag\\ && -\frac{\alpha}{2}  \|\nabla f_i(x^k)\|^2 -\frac{\alpha}{2}  \|\nabla f(x^k)\|^2 \notag\\
&=& \left(\frac{\gamma_{\bound}}{2}-\frac{\alpha}{2} \right)  \|\nabla f_i(x^k)\|^2+\left(\frac{\gamma_{\bound}}{2}-\frac{\alpha}{2} \right) \|\nabla f(x^k)\|^2 \notag\\ && - \gamma_{\bound}\left\langle\nabla f(x^k) , \nabla f_i(x^k) \right\rangle
\end{eqnarray}

By the smoothness of function $f$ we have that 
$$f(x^{k+1}) \leq f(x^k) + \langle\nabla f(x^k) , x^{k+1}-x^k \rangle + \frac{L}{2}\|x^{k+1}-x^k\|^2.$$
Combining this with the update rule of SGD we obtain:
\begin{eqnarray}
\label{iaokalkdao2}
f(x^{k+1})& \leq & f(x^k) + \left\langle\nabla f(x^k) , x^{k+1}-x^k \right\rangle + \frac{L}{2}\|x^{k+1}-x^k\|^2\notag\\
& \leq & f(x^k) - \gamma_k \left\langle\nabla f(x^k) , \nabla f_i(x^k) \right\rangle + \frac{L \gamma_k^2}{2}\|\nabla f_i(x^k)\|^2 \notag\\
& \overset{\eqref{SPLRmax}}{\leq} & f(x^k) - \gamma_k \left\langle\nabla f(x^k) , \nabla f_i(x^k) \right\rangle + \frac{L \gamma_{\bound}^2}{2}\|\nabla f_i(x^k)\|^2 \notag\\
& \overset{\eqref{iaokalkdao1}}{\leq} & f(x^k)+\left(\frac{\gamma_{\bound}}{2}-\frac{\alpha}{2} + \frac{L \gamma_{\bound}^2}{2}\right)  \|\nabla f_i(x^k)\|^2+\left(\frac{\gamma_{\bound}}{2}-\frac{\alpha}{2} \right) \|\nabla f(x^k)\|^2 \notag\\ && - \gamma_{\bound}\left\langle\nabla f(x^k) , \nabla f_i(x^k) \right\rangle 
\end{eqnarray}
By taking expectation condition on $x^k$:
\begin{eqnarray}
\Exp_if(x^{k+1})& \overset{\eqref{iaokalkdao2}}{\leq} & f(x^k)+\left(\frac{\gamma_{\bound}}{2}-\frac{\alpha}{2} + \frac{L \gamma_{\bound}^2}{2}\right)  \Exp_i\|\nabla f_i(x^k)\|^2+\left(\frac{\gamma_{\bound}}{2}-\frac{\alpha}{2} \right) \|\nabla f(x^k)\|^2 \notag\\ && - \gamma_{\bound}\left\langle\nabla f(x^k) , \Exp_i\nabla f_i(x^k) \right\rangle \notag\\
&=& f(x^k)+\left(\frac{\gamma_{\bound}}{2}-\frac{\alpha}{2} + \frac{L \gamma_{\bound}^2}{2}\right)  \Exp_i\|\nabla f_i(x^k)\|^2+\left(\frac{\gamma_{\bound}}{2}-\frac{\alpha}{2} \right) \|\nabla f(x^k)\|^2 \notag\\ && - \gamma_{\bound}\|\nabla f(x^k)\|^2 \notag\\
&=& f(x^k)+\left(\frac{\gamma_{\bound}}{2}-\frac{\alpha}{2} + \frac{L \gamma_{\bound}^2}{2}\right)  \Exp_i\|\nabla f_i(x^k)\|^2-\left(\frac{\gamma_{\bound}}{2}+\frac{\alpha}{2} \right) \|\nabla f(x^k)\|^2
\end{eqnarray}
Since $0<\alpha\leq \gamma_{\bound}$ we have that $\left(\frac{\gamma_{\bound}}{2}-\frac{\alpha}{2} + \frac{L \gamma_{\bound}^2}{2}\right)>0$. Thus, we are able to use \eqref{SGC}:
\begin{eqnarray*}
\Exp_if(x^{k+1})&\leq& f(x^k)+\left(\frac{\gamma_{\bound}}{2}-\frac{\alpha}{2} + \frac{L \gamma_{\bound}^2}{2}\right)  \Exp_i\|\nabla f_i(x^k)\|^2-\left(\frac{\gamma_{\bound}}{2}+\frac{\alpha}{2} \right) \|\nabla f(x^k)\|^2\notag\\
&\overset{\eqref{SGC}}{\leq} & f(x^k)+\left(\frac{\gamma_{\bound}}{2}-\frac{\alpha}{2} + \frac{L \gamma_{\bound}^2}{2}\right)  \left[\rho \|\nabla f(x)\|^2 +\delta \right]-\left(\frac{\gamma_{\bound}}{2}+\frac{\alpha}{2} \right) \|\nabla f(x^k)\|^2\notag\\
&=& f(x^k)+\frac{1}{2}\left[\left(\gamma_{\bound}-\alpha+L \gamma_{\bound}^2\right)\rho-\left(\gamma_{\bound}+\alpha\right)\right] \|\nabla f(x)\|^2 +\frac{1}{2}\left(\gamma_{\bound}-\alpha+L \gamma_{\bound}^2\right)\delta\notag\\
\end{eqnarray*}

By rearranging  and taking expectation again:
\begin{equation*}
\underbrace{\left[\left(\gamma_{\bound}+\alpha\right)-\left(\gamma_{\bound}-\alpha+L \gamma_{\bound}^2\right)\rho\right]}_{\zeta} \Exp\left[\|\nabla f(x)\|^2\right]\leq 2\left(\Exp [f(x^k)] -\Exp [f(x^{k+1})]\right)+\left(\gamma_{\bound}-\alpha+L \gamma_{\bound}^2\right)\delta\notag\\
\end{equation*}

If $\zeta>0$ then:
\begin{eqnarray}
 \Exp[\|\nabla f(x^k)\|^2] 
& \leq &\frac{2}{\zeta}\left(\Exp [f(x^k)]  -\Exp[f(x^{k+1})]\right)+\frac{\left(\gamma_{\bound}-\alpha+L \gamma_{\bound}^2\right)\delta}{\zeta}
\end{eqnarray}

By summing from $k=0$ to $K-1$ and dividing by $K$:
\begin{eqnarray}
 \frac{1}{K} \sum_{k=0}^{K-1}\Exp[\|\nabla f(x^k)\|^2] 
& \leq &\frac{2}{\zeta} \frac{1}{K} \sum_{k=0}^{K-1} \left(\Exp [f(x^k)]  -\Exp[f(x^{k+1})]\right)+\frac{1}{K} \sum_{k=0}^{K-1} \frac{\left(\gamma_{\bound}-\alpha+L \gamma_{\bound}^2\right)\delta}{\zeta}\notag\\
& \leq &\frac{2}{\zeta} \frac{1}{K} \left(f(x^0)  -\Exp[f(x^{K})]\right)+\frac{\left(\gamma_{\bound}-\alpha+L \gamma_{\bound}^2\right)\delta}{\zeta}\notag\\
& \leq &\frac{2}{\zeta K} \left(f(x^0)  -f(x^*)\right)+ \frac{\left(\gamma_{\bound}-\alpha+L \gamma_{\bound}^2\right)\delta}{\zeta}
\end{eqnarray}

In the above result we require that $\zeta=\left(\gamma_{\bound}+\alpha\right)-\left(\gamma_{\bound}-\alpha+L \gamma_{\bound}^2\right)\rho > 0$. In order for this to hold we need to make extra assumptions on the values of $\gamma_{\bound}$ and parameter $c$. This is what we do next.

Let us divide the analysis into two cases. That is:
\begin{itemize}
\item(i) If $\frac{1}{2cL_{\max}}\leq \gamma_{\bound}$ then $\alpha=\min\left\{\frac{1}{2cL_{\max}},\gamma_{\bound}\right\}=\frac{1}{2cL_{\max}}$ and 
$$\zeta=\left(\gamma_{\bound}+\alpha\right)-\left(\gamma_{\bound}-\alpha+L \gamma_{\bound}^2\right)\rho= \left(\gamma_{\bound}+\frac{1}{2cL_{\max}}\right)-\left(\gamma_{\bound}-\frac{1}{2cL_{\max}}+L \gamma_{\bound}^2\right)\rho.$$
By solving the quadratic expression of $\zeta$ with respect to $\gamma_{\bound}$, it can be easily shown that $\zeta>0$ if $$0<\gamma_{\bound}< \bar{\gamma_{\bound}}\eqdef \frac{-(\rho-1)+\sqrt{(\rho-1)^2+\dfrac{4L\rho(\rho+1)}{2cL_{\max}}}}{2L\rho}.$$ To avoid contradiction the inequality $\frac{1}{2cL_{\max}}<\bar{\gamma_{\bound}}$ needs to be true, where $\bar{\gamma_{\bound}}$ is the above upper bound of $\gamma_{\bound}$. This is the case of $c>\frac{L\rho}{4L_{\max}}$ which is the assumption of Theorem~\ref{TheoremNonConvex}.
\item (ii) If $ \gamma_{\bound}\leq\frac{1}{2cL_{\max}}$ then $\alpha=\min\left\{\frac{1}{2cL_{\max}},\gamma_{\bound}\right\}=\gamma_{\bound}$ and 
$$\zeta=\left(\gamma_{\bound}+\alpha\right)-\left(\gamma_{\bound}-\alpha+L \gamma_{\bound}^2\right)\rho=\left(\gamma_{\bound}+\gamma_{\bound}\right)-\left(\gamma_{\bound}-\gamma_{\bound}+L \gamma_{\bound}^2\right)\rho=2\gamma_{\bound}-L \gamma_{\bound}^2\rho$$
In this case, by preliminary computations, it can be shown that $\zeta>0$ if  $\gamma_{\bound} < \frac{2}{L \rho}$. For $c>\frac{L\rho}{4L_{\max}}$ it also holds that $\frac{1}{2cL_{\max}}<\frac{2}{L \rho}$.
\end{itemize}
\end{proof}

\subsubsection{Additional convergence result for nonconvex smooth functions: Assuming independence of step-size and stochastic gradient}
Let us now present an extra theoretical result in which we assume that the step-size $\gamma_k$ and the stochastic gradient $\nabla f_i(x^k)$ in each step are not correlated. Such assumption has been recently used to prove convergence of SGD with the AdaGrad step-size \citep{ward2019adagrad} and for the analysis of stochastic line search in \citet{vaswani2019painless}. From a technical viewpoint, we highlight that for the proofs in the non-convex setting, we use the lower and upper bound of SPS rather than its exact form. This is what allows us to use this independence.

Let us state the main Theorem with the extra condition of independence and present its proof.
\begin{theorem}
\label{TheoremNonConvexExtra}
Let $f$ and $f_i$ be smooth functions and assume that there exist $\rho, \delta>0$ such that the condition \eqref{SGC} is satisfied. Assuming independence of the step-size $\gamma_k$ and the stochastic gradient $\nabla f_{i_k}(x^k)$ at every iteration $k$, SGD with SPS$_{\max}$ with $c > \frac{\rho L}{4 L_{\max}}$ and $\gamma_{\bound} < \max \left\{\frac{2}{L\rho}, \frac{1}{\sqrt{\rho c \, L L_{\max}}}\right\}$ converges as:
\begin{align*}
\min_{k \in [K]}\Exp \|\nabla f(x^k)\|^2 \leq  \frac{f(x^0) - f(x^*)}{\alpha \, K} + \frac{L \delta \gamma_{\bound}^2}{2 \alpha} 
\end{align*}
where $\beta_1 = 1 - \frac{\rho c \, L L_{max} \, \gamma_{\bound}^{2}}{2}$ and $\beta_2 = 1 - \frac{\rho L \, \gamma_{\bound}}{2}$, $\alpha = \min \left\{\frac{\beta_1}{2 c L_{\max}}, \gamma_{\bound}\, \beta_2 \right\}$. 
\end{theorem}

\begin{proof}
By the smoothness of function $f$ we have that 
$$f(x^{k+1}) \leq f(x^k) + \langle\nabla f(x^k) , x^{k+1}-x^k \rangle + \frac{L}{2}\|x^{k+1}-x^k\|^2.$$

Combining this with the update rule of SGD we obtain:
\begin{eqnarray}
f(x^{k+1})& \leq & f(x^k) + \left\langle\nabla f(x^k) , x^{k+1}-x^k \right\rangle + \frac{L}{2}\|x^{k+1}-x^k\|^2\notag\\
& \leq & f(x^k) - \gamma_k \left\langle\nabla f(x^k) , \nabla f_{i_k}(x^k) \right\rangle + \frac{L \gamma_k^2}{2}\|\nabla f_i(x^k)\|^2 
\end{eqnarray}
Taking expectations with respect to $i_k$ and noting that $\gamma_k$ is independent of $\nabla f_i(x^k)$ yields:
\begin{eqnarray}
\label{naklsoama30}
\Exp_{i_k}f(x^{k+1})& \leq & f(x^k) - \gamma_k \left\langle\nabla f(x^k) , \Exp_i\nabla f_{i_k}(x^k) \right\rangle + \frac{L \gamma_k^2}{2}\Exp_{i_k}\|\nabla f_i(x^k)\|^2 \notag\\
& = & f(x^k) - \gamma_k \|\nabla f(x^k) \|^2 + \frac{L \gamma_k^2}{2}\Exp_{i_k}\|\nabla f_i(x^k)\|^2 \notag\\
&  \overset{\eqref{SGC}}{\leq}& f(x^k) - \gamma_k \|\nabla f(x^k) \|^2 + \frac{L \gamma_k^2}{2}\left[\rho \|\nabla f(x)\|^2 +\delta \right] \notag\\
& \leq& f(x^k) -\min\left\{\frac{1}{2cL_{\max}},\gamma_{\bound}\right\} \|\nabla f(x^k) \|^2 + \frac{L \gamma_{\bound}^2 }{2}\left[\rho \|\nabla f(x)\|^2 +\delta \right] \notag\\
& =& f(x^k) - \left(\min\left\{\frac{1}{2cL_{\max}},\gamma_{\bound}\right\}-\frac{L \gamma_{\bound}^2 \rho}{2}\right) \|\nabla f(x^k) \|^2 + \frac{L \gamma_{\bound}^2}{2}\delta
\end{eqnarray}
By rearranging and taking expectations again:
\begin{eqnarray}
\underbrace{\left(\min\left\{\frac{1}{2cL_{\max}},\gamma_{\bound}\right\}- \frac{L \gamma_{\bound}^2}{2}\rho\right)}_{\alpha} \Exp[\|\nabla f(x^k)\|^2] 
& \leq &\Exp [f(x^k)]  -\Exp[f(x^{k+1})]+\frac{L \gamma_{\bound}^2}{2}\delta
\end{eqnarray}
Let $\alpha>0$ then:
\begin{eqnarray}
 \Exp[\|\nabla f(x^k)\|^2] 
& \leq &\frac{1}{\alpha}\left(\Exp [f(x^k)]  -\Exp[f(x^{k+1})]\right)+\frac{L \gamma_{\bound}^2\delta}{2\alpha}
\end{eqnarray}
By summing from $k=0$ to $K-1$ and dividing by $K$:
\begin{eqnarray}
 \frac{1}{K} \sum_{k=0}^{K-1}\Exp[\|\nabla f(x^k)\|^2] 
& \leq &\frac{1}{\alpha} \frac{1}{K} \sum_{k=0}^{K-1} \left(\Exp [f(x^k)]  -\Exp[f(x^{k+1})]\right)+\frac{1}{K} \sum_{k=0}^{K-1} \frac{L \gamma_{\bound}^2\delta}{2\alpha}\notag\\
& \leq &\frac{1}{\alpha} \frac{1}{K} \left(f(x^0)  -\Exp[f(x^{K})]\right)+ \frac{L \gamma_{\bound}^2\delta}{2\alpha}\notag\\
& \leq &\frac{1}{\alpha K} \left(f(x^0)  -f(x^*)\right)+ \frac{L \gamma_{\bound}^2\delta}{2\alpha}
\end{eqnarray}
In the above result we require that $\alpha=\left(\min\left\{\frac{1}{2cL_{\max}},\gamma_{\bound}\right\}- \frac{L \gamma_{\bound}^2}{2}\rho\right) > 0$. In order for this to hold we need to make extra assumptions on the values of $\gamma_{\bound}$ and parameter $c$. This is what we do next.

Let us divide the analysis into two cases. That is:
\begin{itemize}
\item(i) If $\frac{1}{2cL_{\max}}\leq \gamma_{\bound}$ then, 
$$\alpha= \left(\min\left\{\frac{1}{2cL_{\max}},\gamma_{\bound}\right\}- \frac{L \gamma_{\bound}^2}{2}\rho\right)= \left(\frac{1}{2cL_{\max}}- \frac{L \gamma_{\bound}^2}{2}\rho\right).$$
By preliminary computations, it can be easily shown that $\alpha>0$ if $\gamma_{\bound}< \frac{1}{\sqrt{L\rho cL_{\max}}}$. To avoid contraction the inequality $\frac{1}{2cL_{\max}}<\frac{1}{\sqrt{L\rho cL_{\max}}}$ needs to be true. This is the case of $c>\frac{L\rho}{4L_{\max}}$ which is the assumptions of Theorem~\ref{TheoremNonConvexExtra}.
\item (ii) If $ \gamma_{\bound}\leq\frac{1}{2cL_{\max}}$ then, 
$$\alpha= \left(\min\left\{\frac{1}{2cL_{\max}},\gamma_{\bound}\right\}- \frac{L \gamma_{\bound}^2}{2}\rho\right)= \gamma_{\bound}- \frac{L \gamma_{\bound}^2}{2}\rho=\gamma_{\bound}\left(1- \frac{L \gamma_{\bound}}{2}\rho\right).$$
In this case, by preliminary computations, it can be shown that $\alpha>0$ if  $\gamma_{\bound} < \frac{2}{L \rho}$. For $c>\frac{L\rho}{4L_{\max}}$ it also holds that $\frac{1}{2cL_{\max}}<\frac{2}{L \rho}$.
\end{itemize}
\end{proof}

\section{Additional Convergence Results}
\label{sec:additional-theory}
In this section we present some additional convergence results.
We first prove a $O(1/\sqrt{K})$ convergence rate of stochastic subgradient method with SPS for non-smooth convex functions in the interpolated setting. Furthermore, similar to~\citet{schmidt201111}, we propose a way to increase the mini-batch size for evaluating the stochastic gradient and guarantee convergence to the optimal solution without interpolation. 

\subsection{Non-smooth Convex Functions}
In all of our previous results we assume that functions $f_i$ are smooth. As a result, in the proofs of our theorems we were able to use the lower bound \eqref{NewBounds} of SPS. In the case that functions $f_i$ are not smooth using this lower is clearly not possible. Below we present a Theorem that handles the case of non-smooth function for the convergence of stochastic subgradient method\footnote{Note that for non-smooth functions, it is required to have stochastic subgradient method instead of SGD. That is, in each iteration we replace the evaluation of $\nabla f_i(x)$ with its subgradient counterpart $g_i(x)$}. For this result we require that a constant $G$ exists such that $\|g_i(x)\|^2<G^2$ for each subgradient of function $f_i$. This is equivalent with assuming that functions $f_i$ are $G$-Lipschitz. To keep the presentation simple we only present the interpolated case. Using the proof techniques from the rest of the paper one can easily obtain convergence for the more general setting.

\begin{theorem}
\label{NonSmooth}
Assume interpolation and that $f$ and $f_i$ are convex non-smooth functions. Let $G$ be a constant such that $ \|g_i(x)\|^2<G^2, \forall i \in [n] \quad \text{and} \quad x \in \R^n$. Let $\gamma^k$ be the subgradient counterpart of SPS \eqref{SPLR} with $c=1$. Then the iterates of the stochastic subgradient method satisfy: 
\begin{equation*}
\Exp \left[f(\bar{x}^K)-f(x^*)\right] \leq \frac{G \|x^0-x^*\|  }{\sqrt{K}}=O\left(\frac{1}{\sqrt{K}}\right)
\end{equation*}
where $\bar{x}^K=\frac{1}{K}\sum_{k=0}^{K-1} x^k.$
\end{theorem}

\begin{proof}
The proof is similar to the deterministic case (see Theorem~\ref{DeterministicPolyakTheorem}). That is, we select the $\gamma_k$ that minimize the right hand side of the inequality after the use of convexity.
\begin{eqnarray}
\|x^{k+1}-x^*\|^2&=&\|x^k-\gamma_k g_i^k-x^*\|^2\notag\\
&=&\|x^k-x^*\|^2-2 \gamma_k \langle x^k-x^*, g_i^k \rangle + \gamma_k^2 \| g_i^k\|^2\notag\\
&\overset{convexity}{\leq}&\|x^k-x^*\|^2-2 \gamma_k \left[f_i(x^k)-f_i(x^*)\right] + \gamma_k^2 \| g_i^k\|^2
\end{eqnarray}

Using the subgradient counterpart of SPS \eqref{SPLR} with $c=1$, that is, $\gamma_k= \frac{f_i(x^k)-f_i(x^*)}{\|g_i^k\|^2}$\footnote{Recall that in the interpolation setting it holds that $f_i^*=f^*=f(x^*)$.} we obtain:
\begin{eqnarray}
\|x^{k+1}-x^*\|^2
&\leq& \|x^k-x^*\|^2-2 \frac{f_i(x^k)-f_i(x^*)}{\|g_i^k\|^2} \left[f_i(x^k)-f_i(x^*)\right] \notag\\ && + \left[\frac{f_i(x^k)-f_i(x^*)}{\|g_i^k\|^2}\right]^2 \| g^k\|^2\notag\\
&=& \|x^k-x^*\|^2- \frac{\left[f_i(x^k)-f_i(x^*)\right]^2}{\|g_i^k\|^2}\notag\\
&\overset{\|g_i(x)\|^2<G^2}{\leq}& \|x^k-x^*\|^2- \frac{\left[f_i(x^k)-f_i(x^*)\right]^2}{G^2}
\end{eqnarray}
taking expectation condition on $x^k$:
\begin{eqnarray}
\Exp_i\|x^{k+1}-x^*\|^2
&\leq& \|x^k-x^*\|^2- \frac{\Exp_i \left[f_i(x^k)-f_i(x^*)\right]^2}{G^2}\notag\\
&\overset{Jensen}{\leq}& \|x^k-x^*\|^2- \frac{\left[\Exp[f_i(x^k)-f_i(x^*)]\right]^2}{G^2}\notag\\
&=& \|x^k-x^*\|^2- \frac{\left[f(x^k)-f(x^*)\right]^2}{G^2}
\end{eqnarray}
Taking expectation again and using the tower property:
\begin{eqnarray}
\Exp\|x^{k+1}-x^*\|^2
&\leq& \Exp\|x^k-x^*\|^2- \frac{\Exp\left[f(x^k)-f(x^*)\right]^2}{G^2}
\end{eqnarray}
By rearranging, summing from $k=0$ to $K-1$ and dividing by $K$:
\begin{eqnarray}
\frac{1}{K}\sum_{k=0}^{K-1} \frac{\Exp\left[f(x^k)-f(x^*)\right]^2}{G^2} 
&\leq& \frac{1}{K}\sum_{k=0}^{K-1} \bigg[ \Exp\|x^k-x^*\|^2-\Exp\|x^{k+1}-x^*\|^2 \bigg]\notag\\
&=& \frac{1}{K}\bigg[ \|x^0-x^*\|^2-\Exp\|x^{K}-x^*\|^2 \bigg]\notag\\
&=& \frac{1}{K}\bigg[ \|x^0-x^*\|^2 \bigg]
\end{eqnarray}
Taking square roots and using Jensen's inequality:
\begin{eqnarray}
\frac{1}{GK} \sum_{k=0}^{K-1} \Exp \left[f(x^k)-f(x^*)\right] \overset{Jensen}{\leq} \frac{1}{G}\sqrt{\frac{1}{K}\sum_{k=0}^{K-1}\Exp\left[f(x^k)-f(x^*)\right]^2 }
&\leq& \frac{1}{\sqrt{K}}\bigg[ \|x^0-x^*\| \bigg]
\end{eqnarray}
Thus,
\begin{eqnarray}
\Exp \left[f(\bar{x}^K)-f(x^*)\right] \overset{Jensen}{\leq} \frac{1}{K} \sum_{k=0}^{K-1} \Exp \left[f(x^k)-f(x^*)\right]
&\leq& \frac{G \|x^0-x^*\|  }{\sqrt{K}},
\end{eqnarray}
where $\bar{x}^K=\frac{1}{K}\sum_{k=0}^{K-1} x^k.$
\end{proof}

\subsection{Increasing Mini-batch Size}
We propose a way to increase the mini-batch size for evaluating the stochastic gradient and guarantee convergence to the optimal solution without interpolation. We present two main Theorems.  In the first Theorem we assume that functions $f_i$ of problem \eqref{MainProb} are $\mu_i$-strongly convex functions and in the second that each function $f_i$ satisfies the PL condition  \eqref{PLcondition} with $\mu_i$ parameter.

\begin{theorem}
\label{TheoremStronglyConvexIncreasingBatch}
Let us have the same assumptions as in Theorem~\ref{TheoremStronglyConvex} and let all $f_i$ be $\mu_i$-strongly convex functions. Then SGD with SPS, and increasing the batch-size progressively such that the batch-size $b_k$ at iteration $k$ satisfies:
\begin{eqnarray}
b_k &\geq& \left[\frac{1}{n} + \frac{1}{4 \gamma_{\bound} \, z^2} \, \frac{\mu_{\min} \mu}{L_{\max}} \, \left(\frac{\| \nabla f(x^k) \|}{L}\right)^2 \right]^{-1}, \notag
\end{eqnarray}
where $z^2 = \sup_{x} \Exp || \nabla f_i (x) - \nabla f(x) ||^2$, converges as:
\begin{eqnarray*}
\Exp\|x^{k}-x^*\|^2
\leq \left(1-\frac{\mu}{4c \, L_{\max}}\right)^k \|x^0-x^*\|^2.
\end{eqnarray*}
\end{theorem}
\begin{proof}
Following the proof of Theorem~\ref{TheoremStronglyConvex} for the batch $b$. From Equation~\ref{eq:sc-inter1},
\begin{eqnarray*}
\Exp_i\|x^{k+1}-x^*\|^2 
&\leq & \left(1-\mu \min\left\{\frac{1}{2cL_{\max}},\gamma_{\bound}\right\} \right) \|x^k-x^*\|^2  + 2\gamma_{\bound} \sigma^2
\end{eqnarray*}
By strong-convexity of all $f_i$, $\forall i \in [n]$ the minibatch function $f_b$ is $\mu_b$-strongly convex and it holds that:
\begin{eqnarray}
\label{cnajsal}
\sigma^2=\Exp \left[f_b(x^*)-f_b^*\right]& \leq & \Exp \left[\frac{1}{2 \mu_b} \| \nabla f_{b} (x^*) \|^2 \right] \notag \\
& \leq & \frac{1}{2 \mu_{\min}} \Exp \| \nabla f_{b} (x^*) \|^2,
\end{eqnarray}
where in the last inequality we use that $\mu_{\min}=\min\{\mu_i\}_{1=1}^n \leq \mu_i \leq \mu_b$.
By the assumption that the gradients at the optimum have bounded variance, from~\citet{harikandeh2015stopwasting, lohr2019sampling}:
\begin{equation}
\label{asdadad}
\Exp \| \nabla f_{b} (x^*) \|^2  \leq  \frac{n - b}{n \, b} z^2 ,
\end{equation}
Thus,
\begin{eqnarray*}
 \sigma^2=\Exp \left[f_b(x^*)-f_b^* \right]& \overset{\eqref{cnajsal},\eqref{asdadad}}{\leq} & \frac{1}{2 \mu_{\min}} \, \frac{n - b}{n \, b }z^2 
\end{eqnarray*}
and as a result,
\begin{eqnarray*}
\Exp \|x^{k+1}-x^*\|^2 & \leq & \left(1-\mu \min \left\{\frac{1}{2cL_{\max}},\gamma_{\bound}\right\} \right) \Exp\|x^k-x^*\|^2 +  \frac{\gamma_{\bound}}{\mu_{\min}} \, \frac{n - b}{n \, b } z^2
\end{eqnarray*}
If we set the batch-size in iteration $k$ such that, 
\begin{eqnarray}
\frac{\gamma_{\bound}}{\mu_{min}} \, \frac{n - b}{n \, b }z^2 & \leq & \frac{\mu}{4 c L_{\max}} \, \left(\frac{\| \nabla f(x^k) \|}{L} \right)^2 \notag \\
\implies b &\geq& \left[\frac{1}{n} + \frac{1}{\gamma_{\bound} \, z^2} \, \frac{\mu_{\min} \mu}{4 c L_{\max}} \, \left(\frac{\| \nabla f(x^k) \|}{L}\right)^2 \right]^{-1} 
\end{eqnarray}
\begin{eqnarray*}
\Exp_i \|x^{k+1}-x^*\|^2 &\leq & \left(1-\mu \min \left\{\frac{1}{2cL_{\max}},\gamma_{\bound}\right\} \right) \|x^k-x^*\|^2  + \frac{\mu}{4 c L_{\max}} \left(\frac{\| \nabla f(x^k) \|}{L} \right)^2 \\
& \leq & \left(1-\mu \min \left\{\frac{1}{2cL_{\max}},\gamma_{\bound}\right\} \right) \|x^k-x^*\|^2  + \frac{\mu}{4 c L_{\max}} \|x^k - x^* \|^2 \\
\Exp_i \|x^{k+1}-x^*\|^2 & \leq & \left(1- \mu \min\left\{\frac{1}{4cL_{\max}},\gamma_{\bound} - \frac{1}{4 c L_{\max}} \right\} \right) \|x^k-x^*\|^2
\end{eqnarray*}
Following the remaining proof of Theorem~\ref{TheoremStronglyConvex}, 
\begin{eqnarray*}
\Exp \|x^{k}-x^*\|^2 & \leq & \left(1- \mu \min\left\{\frac{1}{4cL_{\max}},\gamma_{\bound} - \frac{1}{4 c L_{\max}} \right\} \right)^{k} \|x^0-x^*\|^2
\end{eqnarray*}
For SPS it holds that $\gamma_{\bound} = \infty$. Thus,
\begin{eqnarray}
\Exp \|x^{k}-x^*\|^2 & \leq & \left(1-  \frac{\mu}{4cL_{\max}} \right)^{k} \|x^0-x^*\|^2
\end{eqnarray}
\end{proof}

\begin{theorem}
\label{TheoremNonConvexPLIncreasingBatch}
Assume that all functions $f_i$ satisfy the PL inequality \eqref{PLcondition} and let $f$ and $f_i$ be smooth functions. Then SGD with SPS$_{\max}$, and increasing the batch-size progressively such that the batch-size $b_k$ at iteration $k$ satisfies:
\begin{eqnarray}
b_k &\geq& \left[\frac{1}{n} + \frac{2}{\gamma_{\bound} \, z^2} \, \frac{\mu_{\min} v}{c L} [f(x^k) - f(x^*)] \,  \right]^{-1},
\end{eqnarray}
where $z^2 = \sup_{x} \Exp || \nabla f_i (x) - \nabla f(x) ||^2$, converges as:
\begin{eqnarray}
\Exp[f(x^k) - f(x^*)] \leq \left(1-v/2 \right)^k  \left[ f(x^0)- f(x^*)\right],
\end{eqnarray}
where  $v = 1 - \gamma_{\bound} \left(\frac{1}{\alpha}- 2\mu  + \frac{L_{\max}}{2c}\right) \in (0,1)$. 
\end{theorem}
\begin{proof}
Following the proof of Theorem~\ref{TheoremNonConvexPL}, from Equation~\ref{oakslxa},
\begin{eqnarray*}
\Exp\left[f(x^{k+1})-f(x^*)\right]& \leq & \underbrace{\gamma_{\bound} \left(\frac{1}{\alpha}- 2\mu  + \frac{L_{\max}}{2c}\right)}_{1 - v}\Exp \left[ f(x^k)- f(x^*)\right]+\frac{L \gamma_{\bound}}{2c} \, \Exp \left[f_b(x^*)-f_b^*\right]
\end{eqnarray*}
Similar to the proof of Theorem~\ref{TheoremStronglyConvexIncreasingBatch}, since each function $f_i$ is PL, 
\begin{eqnarray}
\Exp \left[f_b(x^*)-f_b^*\right] & \leq & \Exp \left[\frac{1}{2 \mu_i} \| \nabla f_{b} (x^*) \|^2 \right] \notag \\
& \leq & \frac{1}{2 \mu_{\min}} \Exp \| \nabla f_{b} (x^*) \|^2 \notag \\
\Exp \left[f_b(x^*)-f_b^*\right] & \leq & \frac{1}{2 \mu_{\min}} \, \frac{n - b}{n \, b } z^2
\end{eqnarray}
From the above relations, 
\begin{eqnarray}
\Exp\left[f(x^{k+1})-f(x^*)\right]& \leq & (1 - v) \, \Exp \left[ f(x^k)- f(x^*)\right] + \frac{L \gamma_{\bound}}{2c} \, \frac{1}{2 \mu_{\min}} \, \frac{n - b}{n \, b } z^2 
\end{eqnarray}
If we set the batch-size $b$ s.t. 
\begin{eqnarray*}
\frac{L \gamma_{\bound}}{2c} \, \frac{1}{2 \mu_{\min}} \, \frac{n - b}{n \, b } z^2 & \leq & \frac{v}{2} \left[ f(x^k)- f(x^*)\right] \\
\implies b &\geq& \left[\frac{1}{n} + \frac{2}{\gamma_{\bound} \, z^2} \, \frac{\mu_{\min} v}{c L} [f(x^k) - f(x^*)] \,  \right]^{-1}
\end{eqnarray*}
\begin{eqnarray}
\implies \Exp\left[f(x^{k+1})-f(x^*)\right] & \leq & \left(1-v/2 \right) \Exp \left[ f(x^k)- f(x^*)\right] 
\end{eqnarray}
Following the remaining proof of Theorem~\ref{TheoremNonConvexPL},
\begin{eqnarray}
\Exp\left[f(x^{k})-f(x^*)\right] &\leq& \left(1-v/2 \right)^k  \left[ f(x^0)- f(x^*)\right].
\end{eqnarray}
\end{proof}

\clearpage
\section{Computing $f_i^*$ for $\ell_2$-regularized standard surrogate losses}
\label{app:compute:fi}

In this section, we explain how the values of $f_i^*$ can be computed in closed form expressions for some standard binary surrogate losses from~\citet{bartlett2006convexity} with $\ell_2$-regularization. These closed form expressions are using the Lambert~$W$ function~\citep{corless1996lambertw} or the more general $r$-Lambert function~\citep{mezHo2017generalization}. While these functions have efficient numerical routines to compute them (see e.g.~\citet{corless1996lambertw}), we note that we can also compute easily $f_i^*$ for the cases in this section by solving simple strongly convex minimization problems in 1 dimension (see~\eqref{eq:fi_g} and~\eqref{eq:fi_g_exp} below). This can be done efficiently to machine precision using Newton's method for example, and could be used to pre-compute $f_i^*$ for each $i$ in our synthetic experiments of Section~\ref{sec:experiments-syn-ninter}.

Following \cite{mezHo2017generalization}, we first start by presenting the definition of the Lambert $W$ function and its recent generalization, the $r$-Lambert function. 

\begin{definition}[Lambert $W$ function]
Consider the transcendental equation \begin{equation}\label{extranoaknsa}
xe^x=a.
\end{equation} 
The inverse of the function on the left-hand side of the above equation ($xe^x$) is called the Lambert $W$ function and is denoted by $W$. In general, $W$ is a multivalued function on complex numbers. But for $a \geq 0$, there is a unique real solution (called the principal branch) and this is the one that we will consider in this paper, i.e.\ the unique real solution of~\eqref{extranoaknsa} for $a \geq 0$ is given by $x=W(a)$. We note that $W(a)$ is a strictly increasing function for $a \geq 0$ with $W(0)=0$, and that there are efficient numerical routines to compute it~\citep{corless1996lambertw}.
\end{definition}

\begin{definition}[$r$-Lambert function]
The $r$-Lambert function is a direct generalization of the Lambert $W$ function first proposed by~\citet{mezHo2017generalization}. 
It is used to express the solution of the transcendental equation
\begin{equation}\label{extranoaknsa2}
xe^x +rx =a,
\end{equation}  where $r$ is a fixed real number.
The inverse of the function $x e^x+rx$ is called the $r$-Lambert function and is denoted by $W_r$. Again, there are multiple possible inverses in general; we will consider the principal branch in this paper, which is a strictly increasing function on its domain of definition (which always includes $\R_+$ -- see Theorem~4 in~\citet{mezHo2017generalization}). Thus $x=W_r(a)$ solution of~\eqref{extranoaknsa2}, at least valid for $a \geq 0$
\end{definition}
From the above definitions, it is clear that the classical Lambert $W$ function is a special case of the $r$-Lambert function when $r=0$. In this case we write $W_0$.

Manipulating the transcendental equation~\eqref{extranoaknsa2}, we can easily solve the slightly more general equation as given in the following Theorem.

\begin{theorem}[Variant of Theorem 3 from \citet{mezHo2017generalization}]
\label{LambertTheorem}
Let $c \in \R$. The equation $xe^{cx} +r x =a$ can be resolved by the $r$-Lambert function and the solution can be expressed as:
\begin{equation}
x = \frac{1}{c} W _{r} (c a)
\end{equation}
\end{theorem}

In the rest of this section, we provide the analytical derivations of the closed form expression of $f_i^*$ when the given losses are $\ell_2$-regularized version of: 
\begin{enumerate}
\item the binary log-loss. We show that in this case that $f_i^*$ can be computed in closed form expression using the $r$-Lambert function.
\item the binary exponential loss. As mentioned by~\citet{bartlett2006convexity} this loss appears in the Adaboost algorithm, amongst others. In this case, the Lambert $W$ function is used.
\end{enumerate}

\subsection{Binary Log-loss}
The $\ell_2$-regularized logistic regression problem is given by:
\begin{equation}
\label{logRegression}
f(x)=\frac{1}{n} \sum_{i=1}^n \log \left(1+e^{-b_i \langle A_i , x \rangle} \right) +\frac{\lambda}{2} \|x\|^2,
\end{equation}
where $A_i$ is the input feature vector for the $i^{th}$ datapoint while $b_i \in \{-1,1\}$ is its label, and $\lambda$ is the regularization parameter.

Note that by following the notation of the rest of the paper, in \eqref{logRegression} we have $$f_i (x)=\log \left(1+e^{-b_i \langle A_i , x \rangle} \right) +\frac{\lambda}{2} \|x\|^2 .$$

To simplify the notation, let us define $z_i:=b_i A_i$. Let us also decompose the vector $x$ in its direction $\hat{x}$ (element of the unit sphere, i.e.~$\|\hat{x}\|=1$) and its norm $\alpha=\|x\|$, and thus $x=\alpha \hat{x}$ for $\alpha \in \R$.  Then we obtain the following:
\begin{eqnarray}
\label{cnaodao}
f_i^*\eqdef \inf_{x} f_i(x)& = & \inf_x \log \left(1+e^{- \langle x,z_i \rangle} \right) +\frac{\lambda}{2} \|x\|^2 \notag\\
&=& \inf_\alpha \inf_{\hat{x}} \log \left(1+e^{-\alpha \langle \hat{x},z_i \rangle} \right)+ \frac{\lambda}{2} \alpha^2
\end{eqnarray}	
Note that $\log \left(1+e^{-\alpha \langle \hat{x},z_i \rangle} \right)$ is decreasing as $\langle \hat{x},z_i \rangle$ increases.
Thus, by Cauchy-Schwartz inequality, $\inf_{\hat{x}}$ is reached when $\hat{x}=  \frac{z_i}{\| z_i \|}$ and equation~\eqref{cnaodao} takes the following form:
\begin{eqnarray} \label{eq:fi_g}
f_i^*&=& \inf_\alpha \underbrace{\log \left(1+e^{-\alpha \|z_i\|} \right)+ \frac{\lambda}{2} \alpha^2}_{:= g(\alpha)}
\end{eqnarray}	
$g(\alpha)$ is a strongly convex function of $\alpha$, and we can find its global minimum by setting its gradient to zero: 
$$\nabla g(\alpha)=\frac{-\|z_i\| e^{-\alpha \|z_i\|}}{1+e^{-\alpha \|z_i\|}}+\lambda \alpha=0.$$
Let $c=\|z_i\|$, then by rearranging the last equation, we obtain:
$$\alpha +\alpha  e ^{\alpha c} =\frac{c}{\lambda} .$$

Using Theorem~\ref{LambertTheorem}, the solution of the above equation can be expressed as follows (using $r=1$): 
\begin{equation}
\label{alphastarLogistic}
\alpha^*=\frac{1}{c} W_1 \left( \frac{c^2}{\lambda} \right).
\end{equation}

Thus to get a closed form expression for $f_i^*$, one can plug $\alpha^*$ in~\eqref{eq:fi_g}, i.e.\ $f_i^* = g(\alpha^*)$.

\subsection{Binary Exponential Loss}
The $\ell_2$-regularized binary exponential loss problem is given by:
\begin{equation}
\label{ExpLoss}
f(x)=\frac{1}{n} \sum_{i=1}^n e^{-b_i \langle A_i , x \rangle} +\frac{\lambda}{2} \|x\|^2,
\end{equation}
with the same notation as for the logistic regression problem~\eqref{logRegression}.

From \eqref{ExpLoss} it is clear that,
$$f_i (x)=e^{-b_i \langle A_i , x \rangle } +\frac{\lambda}{2} \|x\|^2.$$

As for the logistic regression derivation, defining $z_i:=b_i A_i$, and letting $x=\alpha \hat{x}$ with $\|\hat{x}\|=1$, then we obtain the following:
\begin{eqnarray}
f_i^*\eqdef \inf_{x} f_i(x)& = & \inf_x e^{- \langle x,z_i \rangle} +\frac{\lambda}{2} \|x\|^2 \notag\\
&=& \inf_\alpha \inf_{\hat{x}} e^{-\alpha \langle \hat{x},z_i \rangle} + \frac{\lambda}{2} \alpha^2 
\end{eqnarray}	
As for the logistic regression, note that by Cauchy-Schwartz inequality, $\inf_{\hat{x}}$ is reached when $\hat{x}= \frac{z_i}{\| z_i \|}$.
Thus,
\begin{eqnarray} \label{eq:fi_g_exp}
f_i^*&=& \inf_\alpha \underbrace{ e^{-\alpha \|z_i\|}+ \frac{\lambda}{2} \alpha^2}_{:=g(\alpha)}
\end{eqnarray}	

Again, $g(\alpha)$ is a strongly convex function of $\alpha$, and we can find its global minimum by setting its gradient to zero:
$$\nabla g(\alpha)=-\|z_i\| e^{-\alpha \|z_i\|}+\lambda \alpha=0.$$
Let $c=\|z_i\|$, then by rearranging the last equation, we obtain:
$$\alpha  e ^{\alpha c} =\frac{c}{\lambda}$$

Using Theorem~\ref{LambertTheorem}, the solution of the above equation can be expressed as follows (using $r=0$): 
\begin{equation}
\label{alphastarExponential}
\alpha^*=\frac{1}{c} W_0 \left( \frac{c^2}{\lambda} \right),
\end{equation}
where $W_0$ is the classical Lambert $W$ function.

Thus to get a closed form expression for $f_i^*$, one can plug $\alpha^*$ in~\eqref{eq:fi_g_exp}, i.e.\ $f_i^* = g(\alpha^*)$.
\newpage
\section{Additional Experiments}
\label{app:additional}
Following the experiments presented in the main paper, we further evaluate the performance of SGD with SPS when training over-parametrized models.

\begin{figure*}[ht]
    \centering
    \includegraphics[width = 0.8\textwidth]{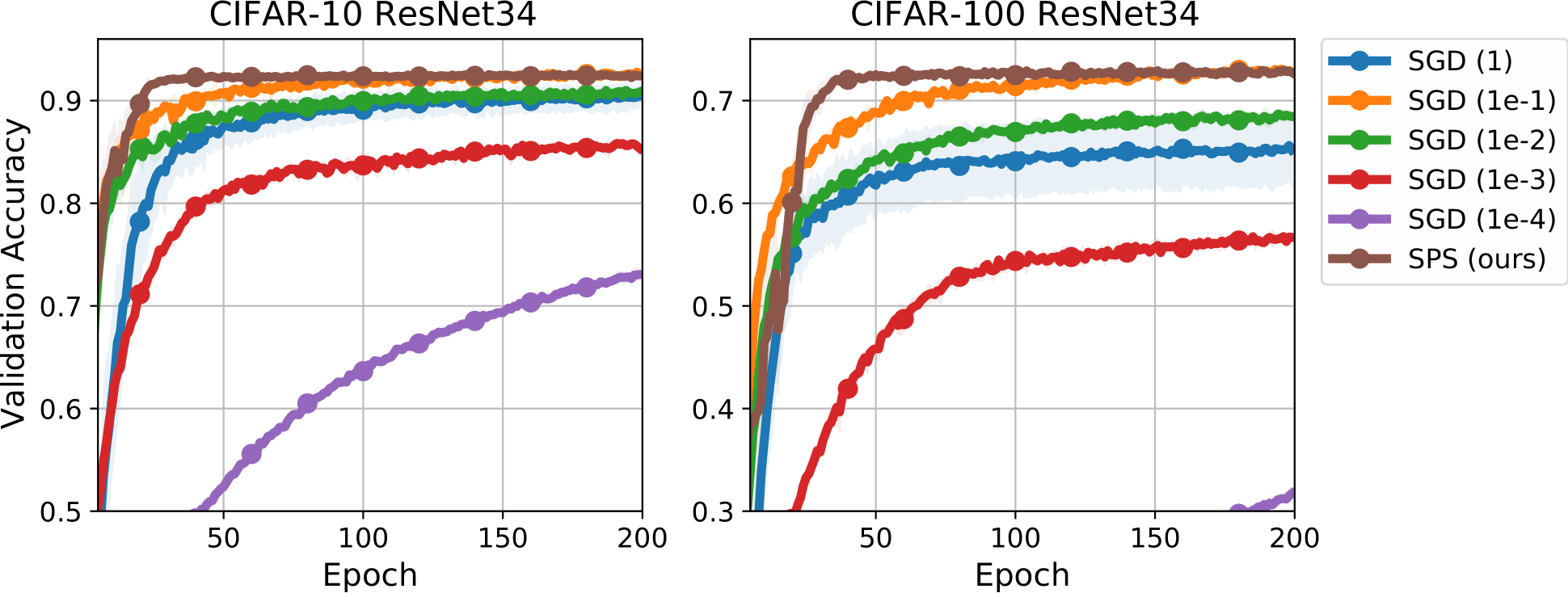}
    \caption{Comparison between SGD with different learning rates and SPS.}
\end{figure*}

\begin{figure*}[ht]
    \centering
    \includegraphics[width = \textwidth]{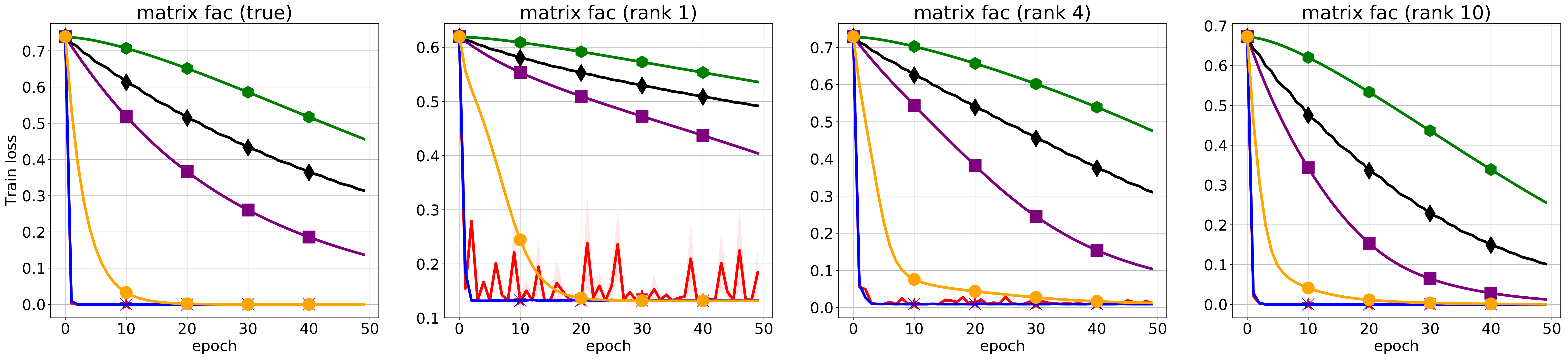}
    \includegraphics[width = \textwidth]{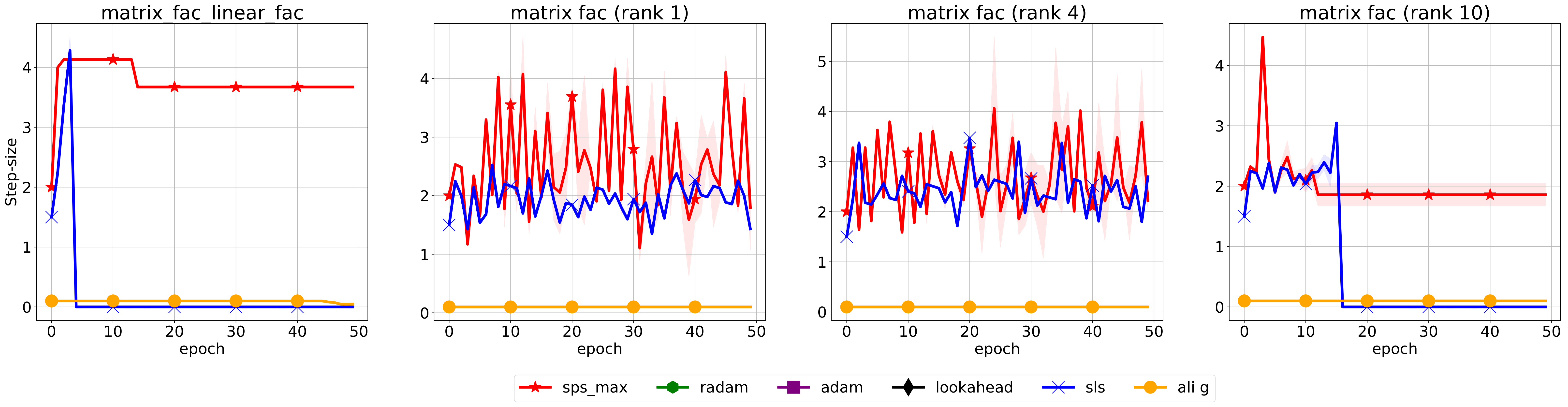}
    \caption{Deep matrix factorization.}
\end{figure*}

\begin{figure*}[ht]
    \centering
    \includegraphics[width = \textwidth]{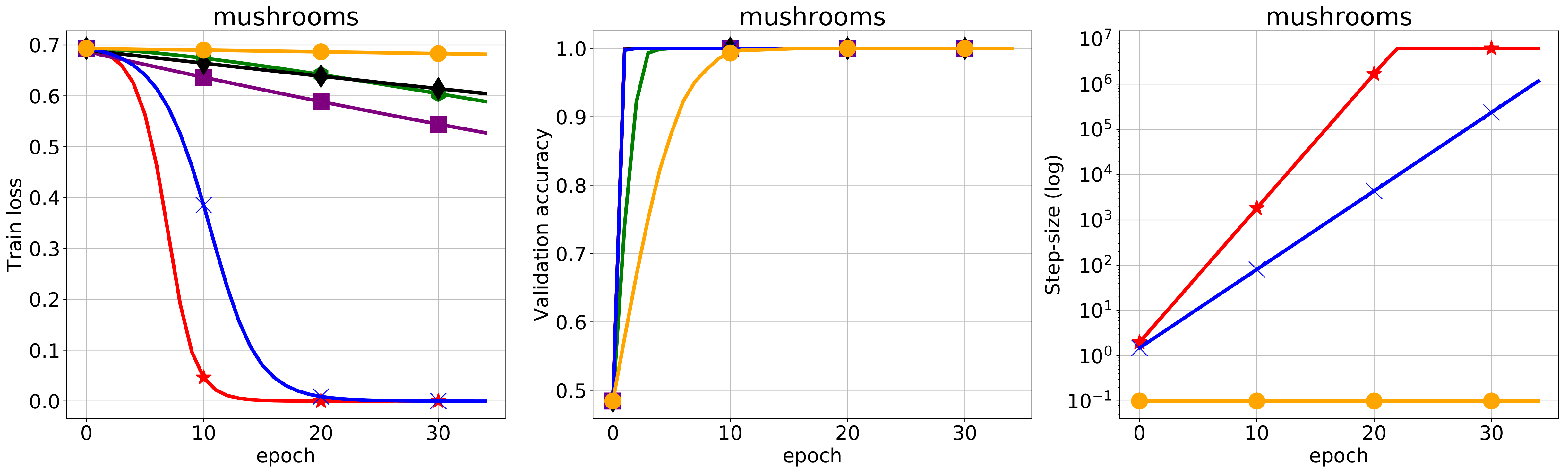}
    \includegraphics[width = \textwidth]{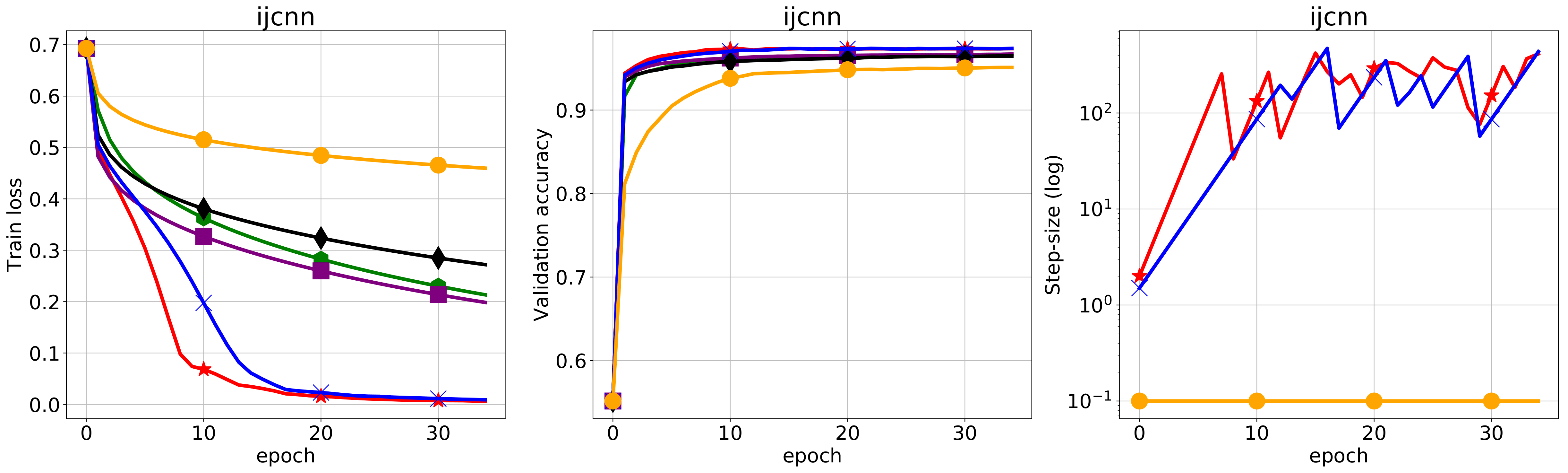}
    \includegraphics[width = \textwidth]{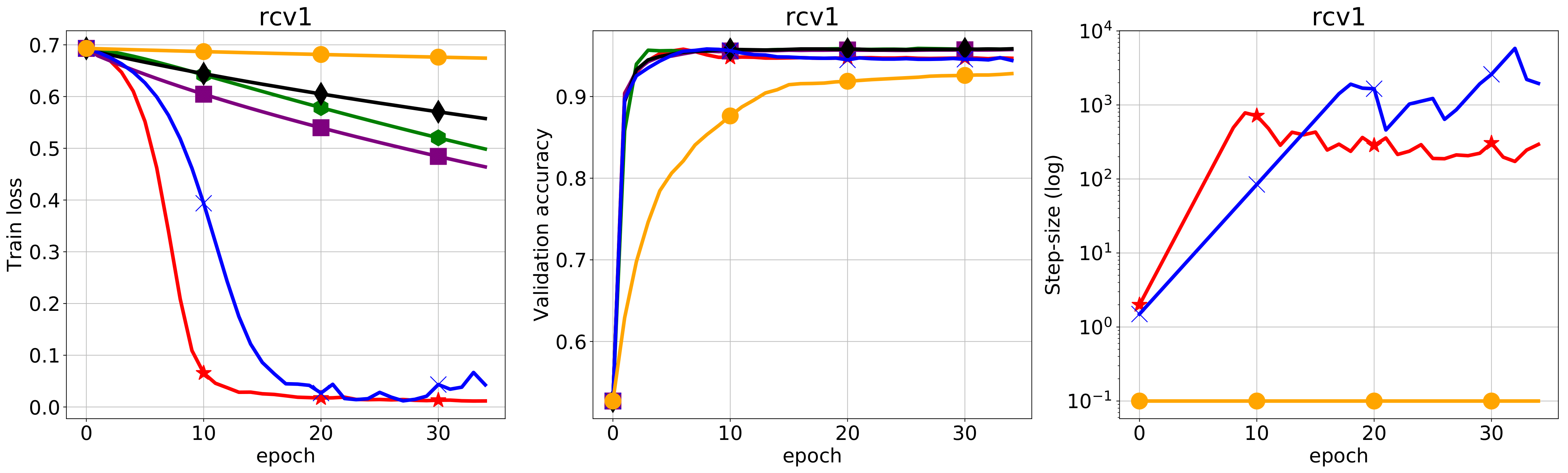}
    \includegraphics[width = \textwidth]{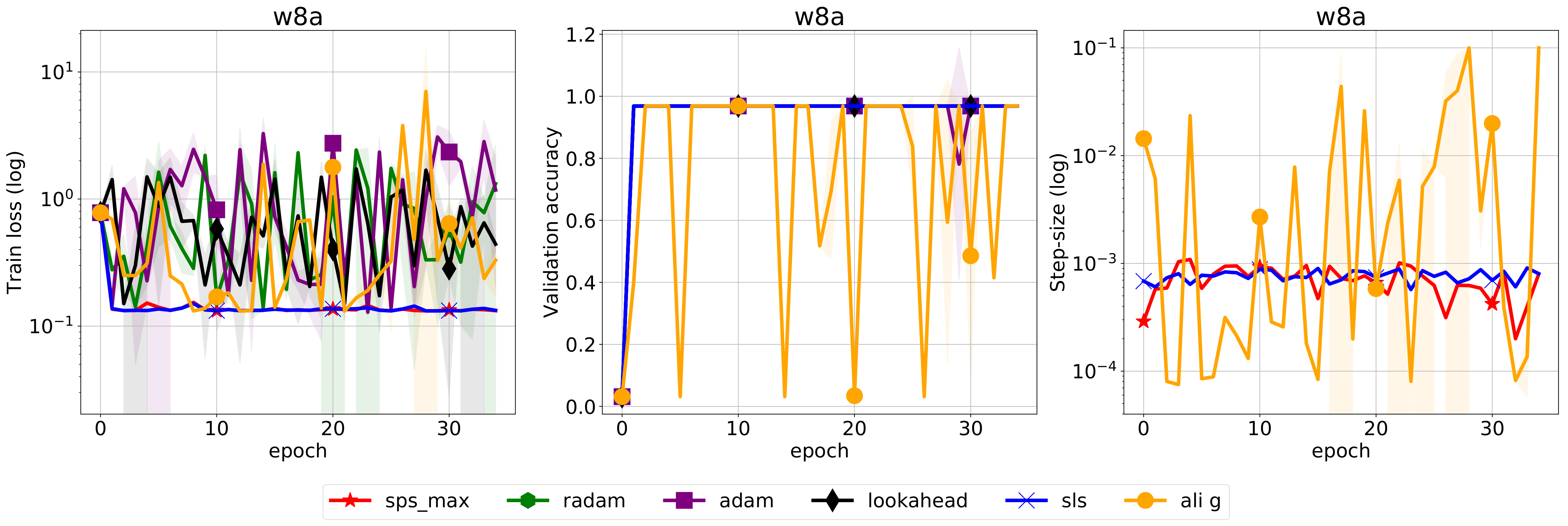}
    \caption{Binary classification using kernels. Data: mushrooms, ijcnn, rcv1, w8a.}
\end{figure*}    

\begin{figure*}[ht]
    \centering
    \includegraphics[width = \textwidth]{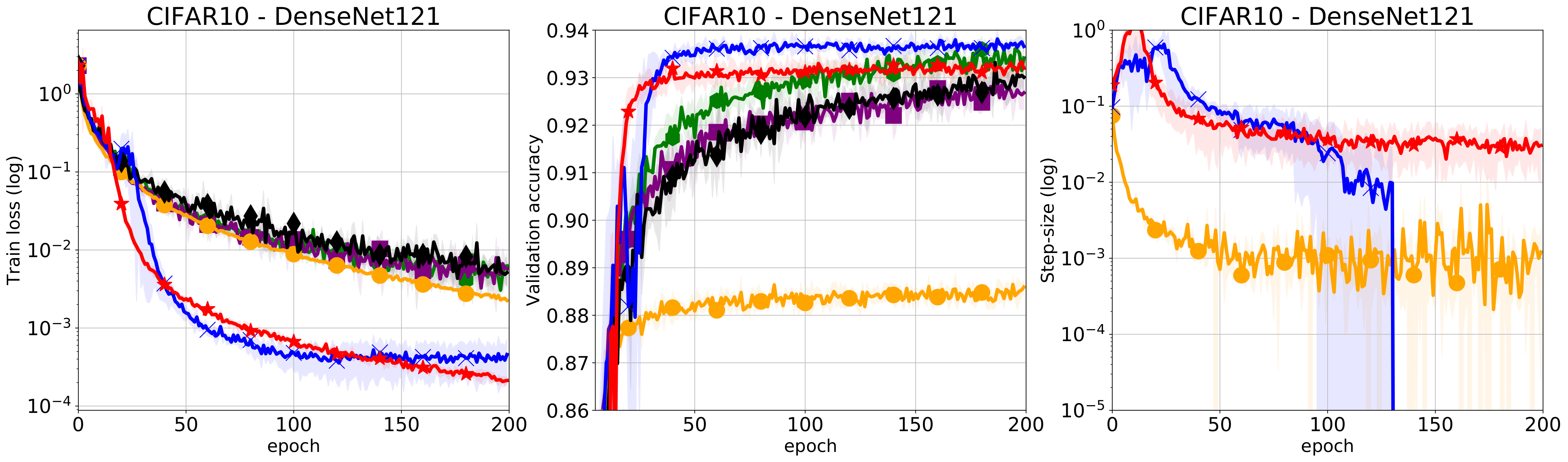}
    \includegraphics[width = \textwidth]{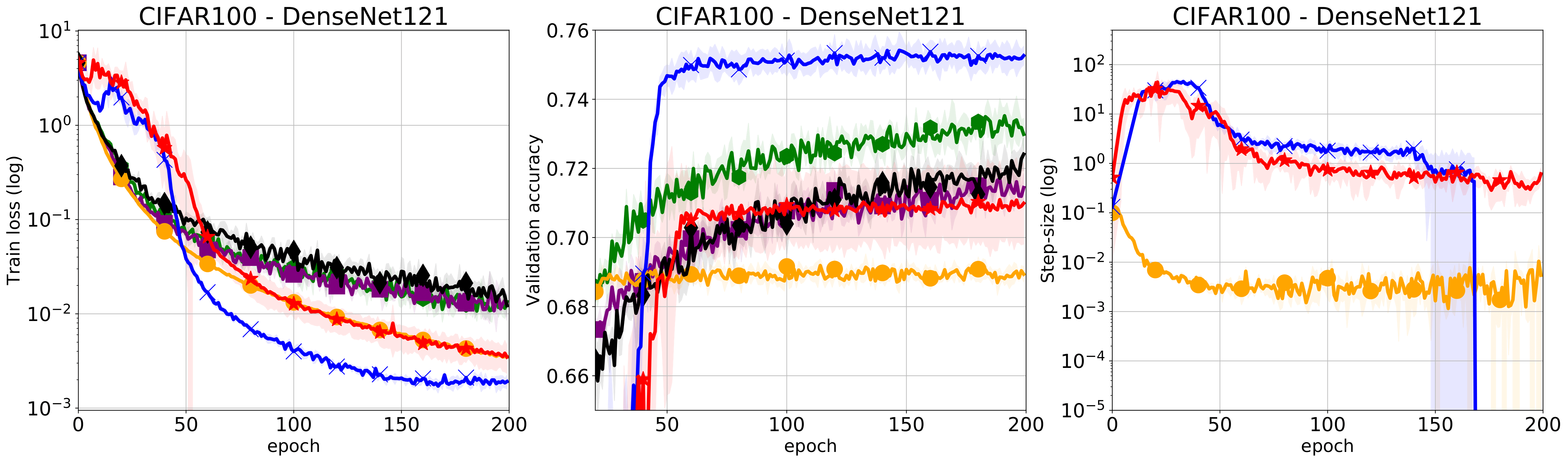}
    \includegraphics[width = \textwidth]{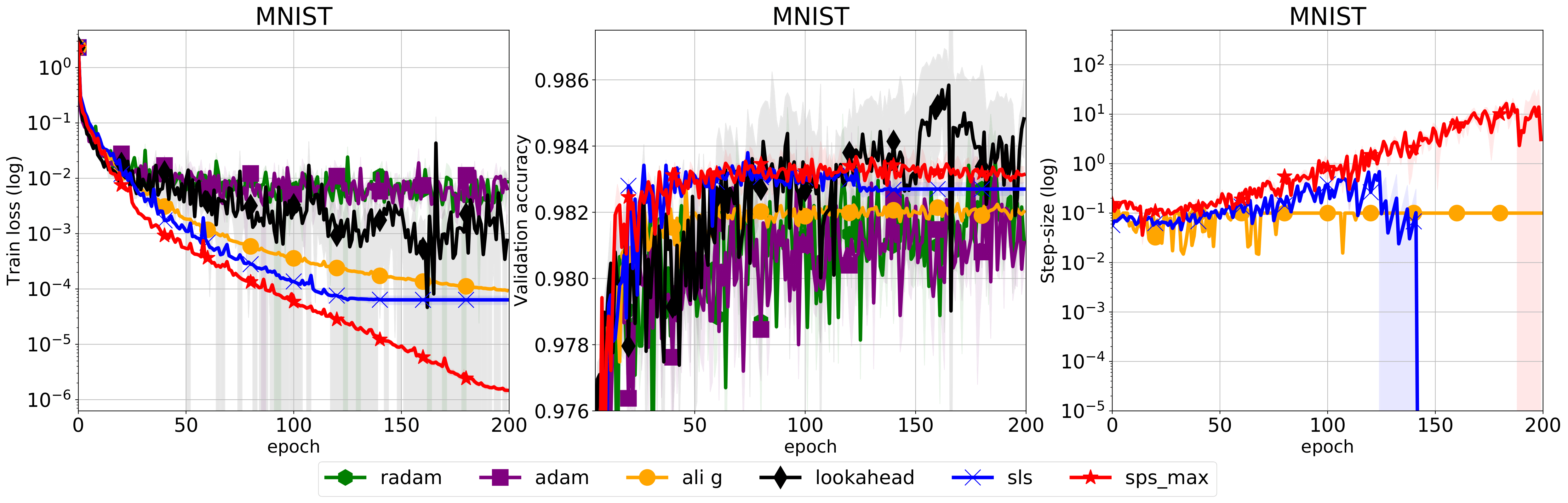}
    \caption{Further experiments on Multi-class classification using deep networks. Setting: CIFAR10-DenseNet121, CIFAR100-DenseNet121 and MNIST.}
\end{figure*}

\end{document}